\documentclass{amsart}
\usepackage{mathtools}
\usepackage{amsfonts}
\usepackage{amssymb}
\usepackage{palatino}
\usepackage{tikz}
\usetikzlibrary{ decorations.markings}
\usepackage{float}
\usepackage{verbatim}
\DeclareMathOperator{\ad}{ad}

\DeclareFontFamily{U}{mathx}{}
\DeclareFontShape{U}{mathx}{m}{n}{<-> mathx10}{}
\DeclareSymbolFont{mathx}{U}{mathx}{m}{n}
\DeclareMathAccent{\widehat}{0}{mathx}{"70}
\DeclareMathAccent{\widecheck}{0}{mathx}{"71}

\theoremstyle{plain}

\newtheorem{theorem}{Theorem}[section]
\newtheorem{corollary}[theorem]{Corollary}
\newtheorem{lemma}[theorem]{Lemma}
\newtheorem{proposition}[theorem]{Proposition}
\newtheorem{problem}[theorem]{Problem}
\theoremstyle{remark}
\newtheorem{remark}[theorem]{Remark}

\numberwithin{equation}{section}

\definecolor{deepgreen}{cmyk}{1,0,1,0.5}

\usepackage{hyperref}
\usepackage{mathtools}

\newcommand{\eps}{\varepsilon}

\newcommand{\norm}[2]{\left\Vert #1\right\Vert_{#2}}

\newcommand{\ba}{\breve{a}}
\newcommand{\bb}{\breve{b}}
\newcommand{\bbC}{\mathbb{C}}
\newcommand{\bbR}{\mathbb{R}}

\newcommand{\dbar}{\overline{\partial}}

\newcommand{\zbar}{\overline{z}}
\newcommand{\mlc}{m^{\text{LC}}}
\newcommand{\mlcT}{\widetilde{m}^{\text{LC}}}
\newcommand{\SigmaLC}{\Sigma^{\textrm{LC}}}
\newcommand{\SigmaLCT}{\widetilde{\Sigma}^{\textrm{LC}}}

\newcommand{\twomat}[4]
{
\left(
	\begin{array}{cc}
		{#1}	&	{#2}	\\[10pt]
		{#3}	&	{#4}
		\end{array}
\right)
}

\newcommand{\diagmat}[2]
{
\left(
	\begin{array}{cc}
		{#1}	&	0	\\
		0		&	{#2}
		\end{array}
\right)
}
\newcommand{\offdiagmat}[2]
{
\left(
	\begin{array}{cc}
		0			&		{#1} 	\\
		{#2}		&		0
		\end{array}
\right)
}
\begin{document}
	\title{Long time asymptotics of a perturbed modified KdV
equation }

\author{Gong Chen}
	\author{Jiaqi Liu}
	\author{Yuanhong Tian}
\address[Chen]{School of mathematics, Georgia institute of technology. 686 Cherry Street
Atlanta, GA }
\email{gc@math.gatech.edu}
\address[Liu]{School of mathematics, University of Chinese Academy of Sciences. No.19 Yuquan Road, Beijing China }
\email{jqliu@ucas.ac.cn}
\address[Tian]{School of mathematics, University of Chinese Academy of Sciences. No.19 Yuquan Road, Beijing China}
\email{1838393768@qq.com}	

\begin{abstract}
We derive full asymptotics of the modified KdV equation (mKdV) with a higher-order perturbative term. We make use of the perturbative theory of infinite-dimensional integrable systems developed by P.
Deift and X. Zhou \cite{DZ-2}, and some new and simpler proofs of certain $L^\infty$ bounds and $L^p$ a priori
estimates developed recently in \cite{CLT}. We show that the perturbed
equation exhibits the same long-time behavior as the completely integrable mKdV.
  
\end{abstract}

\maketitle

\tableofcontents

\section{Introduction}
It is well-known that wave motions in complex media are often influenced by additional energy introduced through related processes. Accordingly, the mathematical models describing such motion are frequently formulated as evolution equations incorporating source-like terms or external forcing. We refer to  \cite{Engel,EnKh} for details on physical backgrounds and derivations for the KdV models as concrete examples. 
In this paper we are interested in the long time asymptotics
of perturbed (forced) modified KdV equations
\begin{equation}
u_t + u_{xxx} - 6u^2u_x= \eps F(u),\qquad (x, t)\in (\bbR, \bbR^+).
\end{equation}
We recall the standard defocussing modified KdV equation (mKdV) 
\begin{equation}
\label{eq:mKdV}
    u_t+u_{xxx}-6u^2u_x=0.
\end{equation}
 Important conserved quantities of mKdV are the mass $M$, energy $H$, and momentum $P$:
 \begin{align}
     \label{conserve}
     M=\int_\bbR u^2 dx, \quad H=\int_{\mathbb{R}} \left|\partial_x u\right|^2+|u|^4 d x, \quad P=\int_\bbR u dx.
 \end{align}
For concreteness and simplicity, in this article, we consider the model with an external force is given by a power nonlinearity of  the dependent variable:
\begin{equation}
\label{pmKdV}
u_t + u_{xxx} - 6u^2u_x= \eps u^{\ell},\qquad (x, t)\in (\bbR, \bbR^+).
\end{equation}
Here $0<\varepsilon\ll 1$, $\ell$ is some positive constant. 
 It is easy to verify that the $u^\ell$ term on the RHS of \eqref{pmKdV} is a dissipative term if $\ell$ is an odd number. In this paper we will not emphasize whether the perturbative term is dissipative or not. It will become clear in our proof that the long time asymptotics of \eqref{pmKdV} do not rely on dissipation.

\subsection{Background}
In the seminal paper by P.Deift and X.Zhou \cite{DZ-2}, the authors investigated the perturbation of the cubic nonlinear Schr\"odinger equation with higher order nonlinearity:
\begin{equation}\label{eq:pnls}
    iq_t+q_{xx}-2|q|^2q-\eps|q|^\ell q=0
\end{equation}
with $0<\epsilon\ll 1$, $\ell>7/2$. By establishing the existence of wave operator (\cite[(1.17)]{DZ-2}) 
\begin{equation}
   W^{+}(q) \equiv \lim _{t \rightarrow \infty} U_{-t}^{\mathrm{NLS}} \circ U_t^{\varepsilon}(q) 
\end{equation}where $U_t^{\varepsilon}(q)$ is the solution operator of \eqref{eq:pnls} with initial data $q$ and $U_{t}^{\mathrm{NLS}}$ is the solution flow for the free cubic NLS (i.e., $\epsilon=0$ in \eqref{eq:pnls}, 
the authors show that for $0<\varepsilon\ll 1$ whose size depends on the $H^{1,1}$ Sobolev norm of the initial data, solutions to perturbed NLS equation behave as $t\to \infty$ like solutions of the $1$-d cubic NLS equation which is completely integrable. Their work is the first major step towards building a rigorous perturbation theory for infinite-dimensional integrable systems. \cite{DZ-2} relies partly on the $L^p$ \textit{a priori} estimate developed in \cite{DZ-1}. Recently, in \cite{CLT}, we revisit the perturbative method of \cite{DZ-2} to calculate the long-time asymptotics of \cite[Equation (1.4)]{CLT} and obtain several new \textit{a priori} estimates that are more adaptable to other completely integrable equations.

On the other hand, in \cite{CL} we obtained the long time asymptotics of Equation \eqref{eq:mKdV} through the $\overline{\partial}$ version of nonlinear steepest descent method first developed in \cite{DZ93}. Before describing our result, we give a review of the direct and inverse scatting formalism for the mKdV.
\subsection{The inverse scattering for the mKdV}
To describe our approach, we recall that \eqref{eq:mKdV}
is the compatibility condition for the following commutator:
\begin{align}
\label{eq: compat}
    \left[\partial_x-\mathcal{U}, \partial_t-\mathcal{W}\right]=0,
\end{align}
where
\begin{align*}
\mathcal{U}& =-i z \sigma+\left(\begin{array}{cc}
0 & u \\
{u} & 0
\end{array}\right), \\
\mathcal{W} & =(-4i z^3-2izu^2) \sigma +\left(\begin{array}{cc}
0 & 4z^2u+2izu_x-u_{xx}+2u^3 \\
4z^2u-2izu_x-u_{xx}+2u^3 & 0
\end{array}\right)
\end{align*}
and $\sigma$ is the third Pauli matrix
$$ \sigma= \diagmat{1}{-1}.$$
Thus \eqref{eq:mKdV} generates an isospectral flow for the problem
\begin{equation}
\label{L}
\frac{d}{dx} \Psi =- iz \sigma \Psi + U(x) \Psi
\end{equation}
where
$$ 
U(x) = \offdiagmat{u(x)}{{u(x)}}.$$
This is a standard AKNS system. If $u \in L^1(\bbR) $, Equation \eqref{L} admits bounded 
solutions for $z \in \mathbb{R}$.   There exist unique solutions $\Psi^\pm$ of \eqref{L} obeying the  following space asymptotic conditions
$$\lim_{x \to \pm \infty} \Psi^\pm(x,z) e^{-ix z \sigma} = \diagmat{1}{1},$$
and there is a matrix $S(z)$, the scattering matrix, with 
 $\Psi^+(x,z)=\Psi^-(x,z) S(z)$.
The matrix $T(z)$ takes the form
\begin{equation} \label{matrixT}
 S(z) = \twomat{a(z)}{\bb(z)}{b(z)}{\ba(z)} 
 \end{equation}
and  the determinant relation gives
$$ a(z)\ba(z) - b(z)\bb(z) = 1.$$
Combining this with the symmetry relations 
\begin{align} \label{symmetry}
\ba(z)=\overline{a( \zbar )}, \quad \bb(z) = \overline{ b(z)}
\end{align}
we arrive at
$$|a(z)|^2-|b(z)|^2=1$$
and conclude that $a(z)$ is zero-free. By the standard inverse scattering theory, we formulate the reflection coefficient:
\begin{equation}
\label{reflection}
r(z)=\bb(z)/a(z), \quad z\in\bbR.
\end{equation}
In \cite{Zhou98}, it is shown that for $k, j$ integers with $k\geq 0$, $j\geq 1$, the direct scattering map $\mathcal{R}$ maps $H^{k,j}(\bbR)$ onto $H^{j,k}_1=H^{j,k}(\bbR)\cap \lbrace r: \norm{r}{L^\infty} <1\rbrace$ where $H^{j,k}$ norm is given by:
\begin{equation}
\label{sp: weighted}
    \norm{f}{H^{i,j}(\bbR)}
= \left( \norm{(1+|x|^j)f}{2}^2 + \norm{f^{(i)}}{2}^2 \right)^{1/2}
\end{equation}
and the map 
$\mathcal{R}: u\mapsto r$ is Lipschitz continuous. Since we are dealing with the defocussing mKdV, only the reflection coefficient $r$ is needed for the reconstruction of the solution. We then have the following result as a corollary:
\begin{theorem}
\label{thm:biject}
    Given $u_0$ the initial data of Equation \eqref{eq:m3}, then we have the direct scattering map
    \begin{equation}
        \mathcal{R}: H^{2,1}(\bbR)\ni u_0\mapsto r(z)\in H^{1,2}(\bbR).
    \end{equation}
    Moreover, $ \mathcal{R}(u(t))$ evolves linearly
    \begin{equation}
    \label{time-1}
        \mathcal{R}(u(t))=\mathcal{R}(u(0)) e^{8i t z^3}=r(z) e^{8i t z^3}.
    \end{equation}
    And inverse scattering map
     \begin{equation}
        \mathcal{R}^{-1}:  H^{1,2}(\bbR)\ni e^{8i t z^3}r(z)\mapsto u(x,t)\in H^{2,1}(\bbR).
    \end{equation}
   In particular, both $\mathcal{R}$ and $\mathcal{R}^{-1}$ are Lipschitz continuous. 
\end{theorem}
The long-time behavior of the solution to the mKdV equation is obtained through a sequence of transformations of the following RHP:
\begin{problem}
\label{prob:mKdV.RH0}
Given $r(z) \in H^{1,2}(\bbR)$, 
find a $2\times 2$ matrix-valued function $m(z;x,t)$ on $\bbC \setminus \bbR$ with the following properties:
\begin{enumerate}
\item		$m(z;x,t) \to I$ as $|z| \to \infty$.
			\medskip
			
\item		$m(z;,x,t)$ is analytic for $z \in \mathbb{C} \setminus \bbR$ with continuous boundary values
			$$m_\pm(z;x,t) = \lim_{\varepsilon \to 0} m(z\pm i\varepsilon;x,t).$$
			\medskip
			
\item		
The jump relation $m_+(z;x,t) = m_	-(z;x,t) e^{-i\theta \ad \sigma} v(z)$ holds, where
			\begin{equation}
			\label{mKdV.V}
			e^{-i\theta \ad \sigma} v(z)=\twomat{1-|r(z)|^2}
			{-\overline{r(z)}e^{-2i\theta}}{r(z)e^{2i\theta}}
   {1}
			\end{equation}
and the real phase function $\theta$ is given by
\begin{equation}
\label{mKdV.phase}
\theta(z;x,t) = xz+4tz^3
\end{equation}
with stationary points
\begin{equation}
    \label{stationary pt}
   \pm  z_0=\sqrt{ \dfrac{-x}{12t}} .
    \end{equation}
    \end{enumerate}
\end{problem}
From the solution of Problem \ref{prob:mKdV.RH0}, we recover
\begin{align}
\label{mKdV.q}
u(x,t) &= \lim_{z \to \infty} 2i  z m_{21}(x,t,z)
\end{align}
where the limit is taken in $\bbC\setminus \bbR$ along any direction not tangent to $\bbR$.  
\subsubsection{The Beals-Coifman solution} For $z\in \mathbb{R}$, introduce
\begin{equation}
    w_\theta=\left(w^-_\theta, w^+_\theta\right) = \left(\twomat{0}{-\overline{r(z)}e^{-2i\theta}}{0}{0}, \twomat{0}{0}{r(z)e^{2i\theta}}{0}\right)
    \end{equation}
    corresponding to the factorization of the jump matrix:
    \begin{equation}
        v_\theta=\left(v^{-}_\theta\right)^{-1} v^{+}_\theta= \twomat{1}{\overline{r(z)}e^{-2i\theta}}{0}{1}^{-1} \twomat{1}{0}{{r(z)}e^{2i\theta}}{1}.
    \end{equation}
Denote the associated singular integral operator,
\begin{equation}
    C_{v_\theta} h=C^{+}\left(h w^{-}_\theta\right)+C^{-}\left(h w^{+}_\theta\right)
\end{equation}
where  $C^\pm$ is the Cauchy projection:
\begin{equation}
(C^\pm f)(z)= \lim_{z\to \Sigma_\pm}\dfrac{1}{2\pi i} \int_{\Sigma} \dfrac{f(s)}{s-z}ds
\end{equation}
and  $+(-)$ denotes taking limit from the positive (negative) side of the oriented contour.
 Let $\mu\in I+L^2(\bbR)$ solves the equation
\begin{equation}
\label{eq:mu}
     \left(1-C_{v_\theta} \right) \mu=I
\end{equation}
with the resolvent operator satisfying
\begin{equation}
    \norm{(1-C_{v_\theta})^{-1}}{L^2(dz)}\leq (1-\rho)^{-1}
\end{equation}
where $\rho:=\norm{r}{L^\infty(dz)}<1$.
Then a direct calculation shows that
\begin{equation}
    m_{ \pm}=I+C^{ \pm}\left(\mu\left(w^{+}_\theta+w^{-}_\theta\right)\right)=\mu v^\pm_\theta
\end{equation}
is the unique solution to Problem \ref{prob:mKdV.RH0}. Set
\begin{equation}\label{eq:bfU}
    \mathbf{U}(x,t):=\int_{\mathbb{R}} \mu(x, s)\left(w_\theta^{+}(s)+w_\theta^{-}(s)\right) d s.
\end{equation}
Then simple calculation recover the potential
\begin{equation}
  U(x,t)=  \twomat{0}{u(x,t)}{{u(x,t)}}{0}=\frac{1}{2 \pi i} \mathrm{ad} \sigma(\mathbf{U}).
\end{equation}

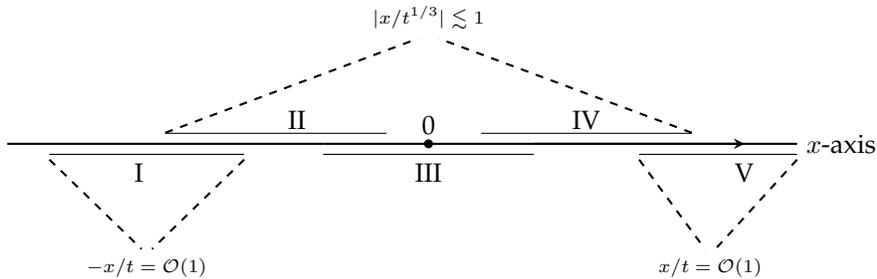
\begin{figure}[h]
\caption{Five Regions}
\vskip 15pt

\begin{tikzpicture}[scale=0.7]
\draw[fill] (0,0) circle[radius=0.075];
\draw[thick] 	(-8,0) -- (-2,0);
\draw[thick]		(-2,0) -- (0,0);
\draw[->,thick,>=stealth]	(0,0) -- (6,0);
\draw[thick]		(2,0) -- (7,0);
\draw		(-7.2, -0.2) -- (-3.5, -0.2); 
\node[below] at (-5.5, -0.2) {I};
\draw		(-5, 0.2) -- (-0.8, 0.2); 
\node[above] at (-2.5, 0.1) {II};

\draw	(-2, -0.2) -- (2, -0.2); 
\node[below] at (0, -.2){III};
\node[above] at (0,.0){0};

\draw	(1, 0.2) -- (5, 0.2); 
\node[above] at (3,0.1) {IV};

\node[right] at (7, 0) {$x$-axis};
\draw (4,-0.2)--(7, -0.2);
\node [below] at (6, -0.2) {V};

\draw[thick, dashed] (-7.2, -0.3)--(-5.45,-2);
\draw[thick, dashed] (-3.5, -0.3)--(-5.25,-2);
\node [below] at(-5.35, -2) {\scriptsize $-x/t=\mathcal{O}(1)$};

\draw[thick, dashed] (-5, 0.21)--(-0.2,2);
\draw[thick, dashed] (5, 0.21)--(0.2, 2);
\node [above] at(0, 2) {\scriptsize $|x/t^{1/3}|\lesssim 1$};

\draw[thick, dashed] (7, -0.3)--(5.45,-2);
\draw[thick, dashed] (4, -0.3)--(5.25,-2);
\node [below] at(5.35, -2) {\scriptsize $x/t=\mathcal{O}(1)$};
\end{tikzpicture}
\label{fig:regions}
\end{figure}
Recall \cite{CL} where we divide the $x$-$t$ space-time region into five parts as illustrated in Figure \ref{fig:regions}. For fixed $M>1$,  $z_0$ given by \eqref{stationary pt} and an auxiliary parameter $\tau=z_0^3 t$, we define the regions as follows:
\begin{itemize}
    \item Region I: $x<0$, $M^{-1}<z_0=\sqrt{ \frac{|x|}{12t}}<M,\,\tau=\left( \frac{|x|}{12t^{1/3}}\right)^{3/2}\gg 1$;
    \item Region II: $x<0$, $M^{-1}\leq \tau=\left( \frac{|x|}{12t^{1/3}}\right)^{3/2} $;
    \item Region III: $\tau=\left( \frac{|x|}{12t^{1/3}}\right)^{3/2}\leq M$;
    \item Region IV: $x>0$, $z_0=\sqrt{ \frac{|x|}{12t}}\leq M, \, \tau=\left( \frac{|x|}{12t^{1/3}}\right)^{3/2}\geq M^{-1}$;
    \item Region V: $x>0$, $ z_0=\sqrt{ \frac{|x|}{12t}}> M^{-1}, \, \tau=\left( \frac{|x|}{12t^{1/3}}\right)^{3/2}\gg 1$.
\end{itemize}
By the $\overline{\partial}$ version of the  nonlinear steepest descent method first obtained in \cite{DMM18}, the following theorem regarding the long time asymptotics of the mKdV equation is obtained in \cite{CL}:
\begin{theorem}
\label{thm:main1}
Given initial data $u_0 \in H^{2,1}(\bbR)$, let $u$ be the solution to the mKdV equation
\begin{equation}
	u_t + u_{xxx} - 6u^2u_x=0 \qquad (x, t)\in (\mathbb{R}, \mathbb{R}^+).
\end{equation}
Let $z_0$ \eqref{stationary pt}, be the stationary point of the phase function \eqref{mKdV.phase}, and define

\begin{equation}
\label{kappa}
\kappa=-\dfrac{1}{2\pi}\log(1-|r(z_0)|^2).
\end{equation}
where $r$ is defined in \eqref{reflection}.
Then we have the following asymptotics
\begin{enumerate}
\item[(i)]  In Region  $\text{I}$, 
\begin{align*}
 u(x,t) &= \left( \dfrac{\kappa}{3tz_0}\right)^{1/2}\cos \left(16tz_0^3-\kappa\log(192tz_0^3)+\phi(z_0) \right) +  \mathcal{O}\left( (z_0 t)^{-3/4}\right) \\
 &=\left(\frac{\kappa}{3 t z_0}\right)^{1 / 2} \cos \left(16 t z_0^3-\kappa \log \left(192 t z_0^3\right)+\phi\left(z_0\right)\right)+\mathcal{O}\left(t^{-\frac{1}{2}}\left(\frac{|x|}{t^{\frac{1}{3}}}\right)^{-\frac{3}{8}}\right)
\end{align*}
where 
\begin{align*}
\phi(z_0) &=\arg \Gamma(i\kappa)-\dfrac{\pi}{4}-\arg r(z_0)\\
        &\quad +\dfrac{1}{\pi}\int_{-z_0}^{z_0}\log\left( \dfrac{1-|r(\zeta)|^2}{1-|r(z_0)|^2} \right)\dfrac{d\zeta}{\zeta-z_0}. 
\end{align*}
\item[(ii)] In Region  $\text{II}$, 
$$u(x,t)=\dfrac{1}{(3t)^{1/3}}P\left( \dfrac{x}{ (3t)^{1/3} } \right)+\mathcal{O} \left(t^{-\frac{1}{2}}\left(\frac{|x|}{t^{\frac{1}{3}}}\right)^{-\frac{3}{8}}\right).$$
\item[(iii)]  In Region  $\text{III}$, 
$$u(x,t)=\dfrac{1}{(3t)^{1/3}}P\left( \dfrac{x}{ (3t)^{1/3} } \right)+\mathcal{O} \left(  t^{ -1/2 } \right).$$

\item[(iv)] In Region  $\text{IV}$, 
\begin{align*}
u(x,t)&=\dfrac{1}{(3t)^{1/3}}P\left( \dfrac{x}{ (3t)^{1/3} } \right)+\mathcal{O} \left( (t\tau)^{-1/2}\right)\\
&=\frac{1}{(3 t)^{1 / 3}} P\left(\frac{x}{(3 t)^{1 / 3}}\right)+\mathcal{O}\left(\left(x^3 t\right)^{-\frac{1}{4}}\right).
\end{align*}
 
\item[(v)] In Region  $\text{V}$, 
$$u(x,t)=\mathcal{O} \left( t^{-1}\right). $$
\end{enumerate}
In the above asymptotics for Regions $\text{II}$, $\text{III}$, $\text{IV}$,  $P$ is a solution of the Painlev\'e $\text{II}$ equation$$P''(s)-sP(s)-2P^3(s)=0$$
determined by $r(0)$. Note that given $r(z)\in H^1(\bbR)$, $r$ is defined pointwise and $r(0)$ makes sense. Also note that in all asymptotics above, the implicit constants in the remainder terms depend only on
$\norm{r}{H^{1}(\bbR)}$.
\end{theorem}
The following corollary is a direct consequence of Theorem \ref{thm:main1}. 
\begin{corollary}\label{cor:pointwise}
    Let $u(x,t)$ be the solution to Equation \eqref{eq:mKdV} with the initial condition $u_0\in H^{2,1}(\bbR)$. Then for  $t>0$,
    $$\norm{u(x,t)}{L^\infty(dx)}< c (1+t)^{-1/3}.$$
\end{corollary}
And this corollary will provide us with the uniform decay estimate of $u(x,t)$ needed in this paper.

\subsection{The inverse scattering for the perturbed mKdV} If we replace the RHS of \eqref{eq: compat} with $\varepsilon G(x)$ where
\begin{equation}
\label{G}
  G(x) =\left(\begin{array}{cc}
0 & u^\ell \\
{u}^\ell & 0
\end{array}\right) 
\end{equation}
then it is easy to check that 
\begin{equation}
    \label{eq: pcompat}
    \left[\partial_x-U, \partial_t-W\right]=\varepsilon G(x)
\end{equation}
is the compatibility condition for \eqref{pmKdV}.  Let $u(x,t)$ be the solution to Equation \eqref{pmKdV}. We then let 
\begin{equation}
    U(x,t)=\twomat{0}{u(x,t)}{{u(x,t)}}{0}
\end{equation}
be the potential of \eqref{L}. Following the standard direct and inverse scattering formalism, we can also formulate the following Riemann-Hilbert problem:
\begin{problem}
\label{prob:mKdV.RH1}
Find a $2\times 2$ matrix-valued function $m(z;x,t)$ on $\bbC \setminus \bbR$ with the following properties:
\begin{enumerate}
\item		$m(z;x,t) \to I$ as $|z| \to \infty$,
			\medskip
			
\item		$m(z;,x,t)$ is analytic for $z \in \mathbb{C} \setminus \bbR$ with continuous boundary values
			$$m_\pm(z;x,t) = \lim_{\varepsilon \to 0} m(z\pm i\varepsilon;x,t),$$
			\medskip
			
\item		
The jump relation $m_+(z;x,t) = m_	-(z;x,t) e^{-ix \ad \sigma} v(z;q(t))$ holds, where
			\begin{equation}
			\label{pmKdV.V}
			e^{-ix\ad \sigma} v(z;u(t))=\twomat{1-|r(z;u(t))|^2}
			 {-\overline{r(z;u(t))}e^{-2ix}}{r(z;u(t))e^{2ix}}
   {1}.
			\end{equation}
\end{enumerate}
\end{problem}
Define
\begin{equation}
    r(t)(z)=e^{-8i t z^3} r(z ; q(t))
\end{equation}
and the following integral equation is originally obtained in \cite{Kaup} and \cite{KN}:
\begin{equation}
\label{eq:r}
    r(t)(z)=r_0(z)+\varepsilon \int_0^t d s e^{-8i z^3 s} \int_{-\infty}^{+\infty} d y e^{-2i y z}\left(m_{-}^{-1}(z ; y,s) G(u(y, s)) m_{-}(z ; y,s)\right)_{21}
\end{equation}
where
\begin{equation}
    r_0(z)=\mathcal{R}\left(u_0\right)(z), \quad u_0=u(t=0)
\end{equation}
and $m_\pm$ is given by (2) of Problem \ref{prob:mKdV.RH1}.
%

\subsection{Main results}
Before we state our main result, we point out that in the proof of Lemma \ref{le:dDF}, in the presence of $\log$ singularities, we need the estimate \eqref{est:dm} which requires that $\norm{r}{H^1}<1/2$. By the mapping property of the direct scattering transform given in Theorem \ref{thm:biject}, we can easily deduce that there exists some $\tilde{c}>0$ such that if the initial data $\norm{u_0}{L^{2,1}}<\tilde{c}$, we have that $\norm{r}{H^1}<1/2$. So we define
\begin{equation}
    H^{2,1}_{\tilde{c}}(\mathbb{R}):=H^{2,1}(\mathbb{R})\cap \lbrace f: \norm{f}{L^{2,1}}<\tilde{c}\rbrace
\end{equation}
and
\begin{equation}
    H^{1,2}_{\scriptscriptstyle 1/2}(\mathbb{R}):=H^{1,2}(\mathbb{R})\cap \lbrace f: \norm{f}{H^1}<\rho<1/2\rbrace.
\end{equation}
With the choice of $\tilde{c}$, we deduce the following proposition:
\begin{proposition}
 \label{thm:biject'}
    Given $u_0$ the initial data of Equation \eqref{pmKdV}, then we have the direct scattering map
    \begin{equation}
        \mathcal{R}: H^{2,1}_{\tilde{c}}(\bbR)\ni u_0\mapsto r(z)\in H^{1,2}_{\scriptscriptstyle 1/2}(\bbR).
    \end{equation}
    \end{proposition}
\begin{theorem}
\label{thm:r}
Assume $u_0=u(x,0)\in H^{2,1}_{\tilde{c}} $ and $\ell>9$. There exists $\varepsilon(\rho, \eta)\ll 1$ such that $u(x,t) \in C\left([0, \infty), H^{2,1}\right)$ solves \eqref{pmKdV} with  $
\varepsilon=\varepsilon(\rho, \eta)$ and $r(t)(z)=e^{-8i t z^3} \mathcal{R}(q(t))(z)$ uniquely solves \eqref{eq:r}. Moreover, we have that for $t\in (0,+\infty)$, $\norm{r(t)}{H^{1,2}}<2\eta$ and $\norm{r(t)}{H^1}<1/2$ .
\end{theorem}
\begin{proof}
    Equation \eqref{eq:r} can be written in the following form:
\begin{equation}
    \label{eq:rF}
    r(t)(z)=r_0+\varepsilon \int_0^t F(s, r(s)) d s
\end{equation}
with
\begin{equation}
\label{F}
    F(z,t;r)=P_{21} \int_{-\infty}^{+\infty} e^{-i\left(2y z+8t z^3\right)} m_{-}^{-1}\left(y, z ; e^{8i z^3 t} r\right) G\left(\mathcal{R}^{-1}\left(e^{8i z^3 t} r\right)\right)(y) m_{-}\left(y, z ; e^{8i z^3 t} r\right) d y
\end{equation}
where $P_{21}$ denotes the projection of $(2\times2)$-matrices onto their $(2,1)$-entries. For each fixed $t$ sufficiently large ,  by Proposition \ref{prop:G-estimate}
\begin{equation}
\label{epsi1}
    \varepsilon <\frac{\ell/3-2}{ C_{F}^{1,2}(\rho, \eta)}
\end{equation}
and 
\begin{equation}
\label{epsi2}
    \varepsilon <\frac{\ell/3-3}{C_{\Delta F}^{1,2}(\rho, \eta) }
\end{equation}
imply that the RHS of \eqref{eq:rF} defines a contraction, whence,  existence and uniqueness of the fixed point follow. Moreover, it is easy to check that
\begin{equation}
\label{epsi3}
    \varepsilon <\frac{\eta}{2}\frac{\ell/3-2}{C_{F}^{1,2}((\rho+1/2)/2, 2\eta)} ,
\end{equation}
and 
\begin{equation}
\label{epsi4}
     \varepsilon <\frac{1-2\rho}{8}\frac{\ell/3-3}{ C_{F}^{1,2}((\rho+1)/2, 2\eta)}
\end{equation}
imply that $\norm{r(t)}{H^{1,2}}<2\eta$ and $\norm{r(t)}{H^1}<(\rho+1/2)/2$. Therefore we choose $\varepsilon=\varepsilon(\rho, \eta) $ satisfying \eqref{epsi1}-\eqref{epsi4}.
\end{proof}
\begin{theorem}
\label{main}
Given $u_0=u(x,0)\in H^{2,1}_{\scriptscriptstyle 1/2}$, if $u(x,t)$ solves Equation \eqref{pmKdV} with $\varepsilon=\varepsilon(\rho, \eta)$ satisfying \eqref{epsi1}-\eqref{epsi4}. Let  $z_0$ be the stationary point of the phase function \eqref{mKdV.phase}, and 
\begin{equation}
\label{kappat}
\kappa(t)=-\dfrac{1}{2\pi}\log(1-|r(t)(z_0)|^2)
\end{equation}
where $r$ is defined in \eqref{reflection}.
Then we have the following asymptotics
\begin{enumerate}
\item[(i)]  In Region  $\text{I}$, 
\begin{align*}
 u(x,t) &= \left( \dfrac{\kappa(t)}{3tz_0}\right)^{1/2}\cos \left(16tz_0^3-\kappa(t)\log(192tz_0^3)+\phi(z_0) \right) +  \mathcal{O}\left( (z_0 t)^{-3/4}\right) \\
 &=\left(\frac{\kappa(t)}{3 t z_0}\right)^{1 / 2} \cos \left(16 t z_0^3-\kappa(t) \log \left(192 t z_0^3\right)+\phi\left(z_0\right)\right)+\mathcal{O}\left(t^{-\frac{1}{2}}\left(\frac{|x|}{t^{\frac{1}{3}}}\right)^{-\frac{3}{8}}\right)
\end{align*}
where 
\begin{align*}
\phi(z_0) &=\arg \Gamma(i\kappa(t))-\dfrac{\pi}{4}-\arg r(t)(z_0)\\
        &\quad +\dfrac{1}{\pi}\int_{-z_0}^{z_0}\log\left( \dfrac{1-|r(t)(\zeta)|^2}{1-|r(t)(z_0)|^2} \right)\dfrac{d\zeta}{\zeta-z_0}. 
\end{align*}
\item[(ii)] In Region  $\text{II}$, 
$$u(x,t)=\dfrac{1}{(3t)^{1/3}}P\left( \dfrac{x}{ (3t)^{1/3} } \right)+\mathcal{O} \left(t^{-\frac{1}{2}}\left(\frac{|x|}{t^{\frac{1}{3}}}\right)^{-\frac{3}{8}}\right).$$
\item[(iii)]  In Region  $\text{III}$, 
$$u(x,t)=\dfrac{1}{(3t)^{1/3}}P\left( \dfrac{x}{ (3t)^{1/3} } \right)+\mathcal{O} \left(  t^{ -1/2 } \right).$$

\item[(iv)] In Region  $\text{IV}$, 
\begin{align*}
u(x,t)&=\dfrac{1}{(3t)^{1/3}}P\left( \dfrac{x}{ (3t)^{1/3} } \right)+\mathcal{O} \left( (t\tau)^{-1/2}\right)\\
&=\frac{1}{(3 t)^{1 / 3}} P\left(\frac{x}{(3 t)^{1 / 3}}\right)+\mathcal{O}\left(\left(x^3 t\right)^{-\frac{1}{4}}\right).
\end{align*}
 
\item[(v)] In Region  $\text{V}$, 
$$u(x,t)=\mathcal{O} \left( t^{-1}\right). $$
\end{enumerate}
In the above asymptotics for Regions $\text{II}$, $\text{III}$, $\text{IV}$,  $P$ is a solution of the Painlev\'e $\text{II}$ equation$$P''(s)-sP(s)-2P^3(s)=0$$
determined by $r(t)(0)$. Note that given $r(t)(z)\in H^1(\bbR)$, $r$ is defined pointwise and $r(t)(0)$ makes sense. Also note that in all asymptotics above, the implicit constants in the remainder terms depend only on
$\norm{r(t)}{H^{1}(\bbR)}$.
\end{theorem}

\begin{proof}
    The asymptotic formula follows from the uniform $H^{1,2}$ estimate of the reflection coefficient $r(t)$ given by Theorem \ref{thm:r} and an application of the nonlinear steepest descent method as is given by
    \cite{CL}.
\end{proof}
From Theorem \ref{thm:r} we deduce that for $t_2>t_1\gg 0$,
\begin{equation}
    \left\|r\left(t_2\right)-r\left(t_1\right)\right\|_{H^{1,2}} \leqslant \frac{ C_{F}^{1,2}((\rho+1)/2, 2\eta)}{\ell/3-2}\left(\frac{1}{\left(1+t_1\right)^{\ell/3-2}}-\frac{1}{\left(1+t_2\right)^{\ell/3-2}}\right),
\end{equation}
and so $\lbrace r(t) \rbrace$ is a Cauchy sequence in $t$ and $r_\infty:=\lim _{t \rightarrow \infty} r(t)$ exists. Then we conclude
\begin{equation}
    \left\|{U}\left(e^{-i z^3 t} r(t)\right)-{U}\left(e^{-i z^3 t} r_{\infty}\right)\right\|_{L^{\infty}(d x)} \leqslant \left\|r(t)-r_{\infty}\right\|_{H^{1,2}}\leqslant \frac{1}{(1+t)^{\ell/3-2}},
\end{equation}
by Lipschitz continuity of the direct/inverse scattering map. So we arrive at the following corollary:
\begin{corollary}
 For $u_0=u(x,0)\in H^{2,1}_{\tilde{c}}$, if  $u(x,t)$ solves Equation \eqref{pmKdV} with $\varepsilon=\varepsilon(\rho, \eta)$ satisfying \eqref{epsi1}-\eqref{epsi4}, then for any $\ell>9$, $u(x,t)$ admits the same asymptotic expansion:
  \begin{equation}
      u(x,t)=\tilde{u}_{as}(x,t)+\tilde{u}_{err}\left(x,t\right)
  \end{equation}
  where the explicit leading term and error term are given by Theorem \ref{thm:main1} with $r$ replaced by $r_\infty$.
\end{corollary}

\begin{remark}
    In a previous paper by Harrop-Griffith \cite{HGR}, the author obtained long time asymptotics of short range perturbation to mKdV \cite[(1.2)]{HGR} under the smallness assumption of the $H^{1,1}$-norm of the initial data. In our paper we only assume the $L^{2,1}$-norm of the initial condition has certain upper bound thus allowing large oscillation of the initial condition.
\end{remark}
\begin{remark}
    Recall that the focussing mKdV equation 
    $$u_t+u_{xxx}+6u^2u_x=0 $$
     admits soliton solution of the form 
    $$ u(x,t)=2 \zeta  \operatorname{sech}\left(-2 \zeta\left(x-4 \zeta^2 t\right)+\omega\right) $$
    and it easy to check that for any $\zeta>0$,
    $$ \int_{\mathbb{R}}2 \zeta  \operatorname{sech}\left(-2 \zeta\left(x-4 \zeta^2 t\right)+\omega\right) dx=\pi. $$
    So by assuming $\norm{u_0}{L^{2,1}}<\pi$, we can rule out the presence of solitons. This enables us to employ the results in the current paper to investigate the perturbation of the focusing mKdV equation. 
\end{remark}
\subsection{Notations}
Throughout the text, constants $c> 0$ are used generically. $c$ always denotes a constant independent of $x, t$.

And $\Delta$ refers to the difference operator acting on some term $h$ that contains the reflection coefficient $r_i$, $i=1,2$:
$$\Delta h=h(r_2)-h(r_1).$$

We always use $\diamond$ to denote the dummy variable.
\section{A priori estimates}
\label{sec:priori}


The goal of this section is to obtain the following key estimate that is sufficient to prove our Theorem \ref{main}. Throughout the section, we will first consider \emph{a priori} estimate  and the estimates for differences of two solutions. For the first part, we assume that we have a solution $r$ satisfying
\begin{equation}\label{eq:assumption}
  \norm{r}{H^{1,2}_{\scriptscriptstyle 1/2}}\leq\eta,\, \text{and}\, \norm{r}{L^\infty}\leq \rho<1/2. 
\end{equation}
For the second part, we assume that\begin{equation}\label{eq:assumption}
  \norm{r_i}{H^{1,2}_{\scriptscriptstyle 1/2}}\leq\eta,\, \text{and}\, \norm{r_i}{L^\infty}\leq \rho<1/2,\, \text{for }\, i=1,2. 
\end{equation} Compared with \cite[Section 6]{DZ-2}, we make use of the newly obtained resolvent operator estimate in \cite[Proposition 2.2]{CLT} and also  estimate on the $\widetilde{L}$ operator given  in Section \ref{sec:dbar}.

First of all, given estimates for $r$, $r_1$ and $r_2$, solving Problem \ref{prob:mKdV.RH1} and applying the nonlinear steepest descent, the corresponding Beals-Coifman solutions, $u$, $u_1$ and $u_2$ satisfying the pointwise decay for $t> 0$
\begin{equation}
\norm{u(x,t)}{L^\infty(dx)}< c t^{-1/3},\, \norm{u_1(x,t)}{L^\infty(dx)}< c t^{-1/3}, \,\norm{u_2(x,t)}{L^\infty(dx)}< c t^{-1/3}
\end{equation}as in Corollary \ref{cor:pointwise}.

\begin{proposition}
\label{prop:G-estimate}
Let $G$  be the perturbation term \eqref{G}. Set
$$
    F=P_{12} \int e^{-2i \theta} m_{-}^{-1} G m_{-} dy
$$
where $P_{12}$ denotes the projection onto the $(2,1)$ entry of the $2\times 2$ matrix-valued  equation \eqref{eq:rF}.
Then $F$ has the following estimates:
\begin{equation}
\label{est:F}
    \|F\|_{H^{1,2}_{\scriptscriptstyle 1/2}(dz)}\leq \frac{C^{1,2}_F(\rho, \eta)}{(1+t)^{(\ell-3)/3}}
\end{equation}
\begin{equation}
\label{est:DF}
    \|\Delta F\|_{H^{1,2}_{\scriptscriptstyle 1/2}d(z)}\leq \frac{C^{1,2}_{\Delta F}(\rho, \eta)}{(1+t)^{(\ell-6)/3}}\|\Delta r\|_{H^{1,2}_{\scriptscriptstyle 1/2}}.
\end{equation}
where the two constants $C^{1,2}_F(\rho, \eta)$ and $C^{1,2}_{\Delta F}(\rho, \eta)$ are uniform in $x$ and $t$ and are monotonic in $\rho$ and $\eta$.
\end{proposition}
\begin{proof}
Due to the fact that $\|f\|_{H^{1,2}} = \|f\|_{L^2} + \|\diamond^2f\|_{L^2} +\|\partial_x f\|_{L^2}$, we estimate each norm of the term $F$ and $\Delta F$ separately. \eqref{est:F} follows from Lemma \ref{le:FL2}, Lemma 
\ref{le:zF2} and Lemma \ref{le:dF}. \eqref{est:DF} follows from Lemma \ref{le:DF}, Lemma \ref{le:DzF} and Lemma \ref{le:dDF}.
\end{proof}
\begin{lemma}
\label{le:FL2}
\begin{equation}
\|F\|_{L^2(dz)} \leq \frac{C_{F}^{0,0}(\rho, \eta)}{(1+t)^{(\ell-2)/3}}
    \end{equation}
    where the constant $C^{0,0}_F(\rho, \eta)$ is uniform in $x$ and $t$ and are monotonic in $\rho$ and $\eta$.
\end{lemma}

\begin{proof}
Similar to \cite[(6.9)]{DZ-2}, we decompose $F$ into the following parts:
\begin{align}
\label{eq:decomF}
    F&=P_{12}\int e^{-2i \theta}G dy + P_{12}\int e^{-2i \theta}(m_-^{-1}-I)G dy+P_{12}\int e^{-2i \theta} G (m_--I) dy + P_{12}\int e^{-2i \theta}(m_-^{-1}-I)G (m_--I) dy\\
    \nonumber
    &=F^{(1)}+F^{(2)}+F^{(3)}+F^{(4)}.
\end{align}

By the Plancherel's identity and Minkowski's inequality, and the estimate of $\|u\|_{L^{\infty}} $ given by Lemma \ref{lm:mbound2} and the resolvent bound given in Proposition \ref{prop:new apriori}, we deduce that

\begin{align*}
\|F\|_{L^2(dz)} & \leq  c_1 \norm{u}{L^2(dx)}\|u^{\ell-1}\|_{L^\infty(dx)} + 2c_2 \|u^\ell\|_{L^1(dx)}\|m_--I\|_{L^2(dz) \otimes L^ \infty(dx)} + c_3\|u^\ell\|_{L^1(dx)}\|m_-^{-1}-I\|_{L^4(dz) \otimes L^ \infty(dx)}^2 \\
&\leq c t^{-(\ell-2)/3}.
\end{align*}
We combine all constants $c$ in the proof into $C^{0,0}_F(\rho.\eta)$ and conclude the proof.
\end{proof}

\begin{lemma}
\label{le:zF2}
\begin{equation}
    \|\diamond^2 F\|_{L^2(dz)} \leq =\frac{C^{0,2}_{F}(\rho,\eta)}{(1+t)^{(\ell-2)/3}}
    \end{equation}
    where the constant $C^{0,2}_F(\rho, \eta)$ is uniform in $x$ and $t$ and are monotonic in $\rho$ and $\eta$.
\end{lemma}

\begin{proof}
We use the decomposition of $F$ given by \eqref{eq:decomF}.
By using the estimate of $\|u_{xx}\|_{L^2(dx)}$ given by \eqref{L2 uxx}, we get:
\begin{eqnarray*}
    \norm{\diamond^2 F^{(1)}}{L^2(dz)} =c \|\partial^2_{xx} u^\ell(x)\|_{L^2(dx)} \leq c \|\partial^2_{xx} u(x)\|_{L^2(dx)}\|u^{\ell-1}\|_{L^\infty(dx)}\leq c  (1+t)^{-(\ell-1)/3}.
\end{eqnarray*}

The estimate of $z^2F^{(2)}$ is the same as $z^2F^{(3)}$. 
We first use the following commutating relation:
\begin{align}
   \diamond^2 C^{ \pm}-C^{ \pm} \diamond^2&=-\frac{\diamond}{2 \pi i}\int-\left(\frac{1}{2 \pi i} \int\diamond \right).
   \end{align}
Thus we have $z^2F^{(3)}=I_1+I_2+I_3$ with
\begin{align*}
    \norm{I_{1}}{L^2(dz)} &= \left \Vert \int e^{-2i \theta}  G \left[C^-\diamond^2 \mu (w_++w_-)\right]dy \right \Vert_{L^2(dz)}\\
    &\leq c\|G\|_{L^1(dx)} \|\left[C^-\diamond^2 \mu (w_++w_-)\right]\|_{L^2(dz) \otimes L^\infty(dx)} \\
    & \leq c \norm{u}{L^2(dx)}\norm{u}{L^\infty(dx)}^{\ell-1}\norm{\mu}{L^\infty(dz)}\norm{r}{L^{2,1}(dz)}\\
    & \leq c t^{-(\ell-1)/3}.
    \end{align*}
\begin{align*}
   \norm{I_{2}}{L^2(dz)} = \left \Vert \int e^{-2i \theta} z GUdy \right \Vert_{L^2(dz)} \leq&c\|(GU)_x\|_{L^2(dx)} \\
    \leq & c \|u_x\|_{L^2(dx)}\norm{u^{\ell}}{L^\infty(dx)}\\
   \leq & c t^{-\ell/3}.
\end{align*}
\begin{align*}
   \norm{I_{3}}{L^2(dz)} &= \left \Vert \int e^{-2i \theta} G\left(\frac{1}{2 \pi i} \int\diamond \mu (w_++w_-) \right)dy \right \Vert_{L^2(dz)}\\
   &\leq c\|u\|_{L^2(dx)} \norm{\int\diamond \mu (w_++w_-) }{L^\infty(dx)} \|u^{\ell-1}\|_{ L^\infty(dx)} \\
   & \leq c \|u\|_{L^2(dx)} \|\mu-I\|_{L^2(dz)\otimes{L^\infty(dx)}}\norm{r}{L^{2,1}(dz)}\|u\|_{ L^\infty(dx)}^{\ell-1}\\
   &\leq c t^{-(\ell-1)/3}.
   \end{align*}
Finally, we notice that
\begin{align}
    \norm{C^-\left[\diamond \mu (w_++w_-)\right]}{L^4(dz)}&\leq \norm{\diamond \mu (w_++w_-)}{L^4(dz)}\\
    \nonumber
    &\leq \norm{\mu}{L^\infty(dz)}\norm{\diamond r}{L^4(dz)}\\
    \nonumber
    &\leq  \norm{\mu}{L^\infty(dz)}\norm{ r}{L^{2,2}(dz)}\norm{ r}{L^\infty(dz)}
\end{align}
and deduce
\begin{align*}
    \norm{\diamond^2 F^{(4)}}{L^2(dz)} &=\left \Vert \diamond^2 \int e^{-i\theta}(m_-^{-1}-I)G(m_--I)dy \right \Vert_{L^2(dz)}\\
    &\leq \norm{\int e^{-2i\theta }\left( U + C^-\left[\diamond \mu (w_++w_-)\right] \right)G\left( U + C^-\left[\diamond \mu (w_++w_-)\right] \right)}{L^2(dz)} \\
    &\leq \|UGU\|_{L^2(dx)}+2\norm{\int e^{-2i\theta} UG C^-\left[\diamond \mu (w_++w_-)\right] }{L^2(dz)} \\
    & \quad + \norm{\int e^{-2i\theta} C^-\left[\diamond \mu (w_++w_-)\right]G C^-\left[\diamond \mu (w_++w_-)\right] }{L^2(dz)} \\
    &\leq c \norm{u^{\ell+2}}{L^2(dx)} +c\norm{u^{\ell+1}}{L^1(dx)}\norm{\mu-I}{L^\infty(dx)\otimes{L^\infty(dx)}}\norm{r}{L^{2,1}(dx)}\\
    &\quad +  c\norm{u^{\ell}}{L^1(dx)}\norm{C^-\left[\diamond \mu (w_++w_-)\right]}{L^4(dz)}^2\\
    &\leq c t^{-(\ell-2)/3}.
\end{align*}
We combine all the constants $c$ in the proof into $C^{0,2}_F(\rho.\eta)$ and conclude the proof.
\end{proof}

\begin{lemma}
\label{le:dF}
    \begin{equation}
        \|\partial_z F\|_{L^2(dz)} \leq \frac{C^{1,0}(\rho,\eta)}{(1+t)^{(\ell-4)/3}} 
    \end{equation}
    where the constant $C^{1,0}_{\Delta F}(\rho, \eta)$ is uniform in $x$ and $t$ and are monotonic in $\rho$ and $\eta$.
\end{lemma}

\begin{proof}
Define the operator 
\begin{equation}
-\tilde{L}=  ix \text{ad} \sigma + 12tz \partial_x,
\end{equation}
 then we have $\partial_z - \tilde{L} = L - 12tzU$, where $L = \partial_z - i(x + 12 tz^2) \text{ad} \sigma$. 
 We proceed to estimate $(L - 12tzQ)\mu$.

\begin{align*}
    -12itz^2 \text{ad} \sigma C_{w_\theta} \mu & = 12tzU - 12itz \text{ad} \sigma C^+(\diamond \mu w_\theta^-) - 12itz \text{ad} \sigma C^-(\diamond \mu w_\theta^+)\\
    & = 12tzU + 12it\frac{1}{2\pi i} \text{ad} \sigma \int \diamond \mu (w_\theta^- + w_\theta^+) - 12it \text{ad}\sigma C^+(\diamond^2 \mu w_\theta^-) - 12it \text{ad}\sigma C^-(\diamond^2 \mu w_\theta^+),
\end{align*}
\begin{align*}
    -12tzU C_{w_\theta} \mu & = 12tU \frac{1}{2\pi i} \int \mu (w_\theta^- + w_\theta^+) - 12tU C^+(\diamond \mu w_\theta^-) - 12tU C^-(\diamond \mu w_\theta^+),
\end{align*}
\begin{align*}
    (L - 12tzU)\mu = & (L - 12tzU)I + (L - 12tzU) C_{w_\theta} \mu \\
     = & - 12tzU + 12tzU + C_{w_\theta} (L-12tzU) \mu + C_{\omega'_{\theta}} \mu\\
    & + 12it\frac{1}{2\pi i} \text{ad} \sigma \int \diamond \mu (w_\theta^- + w_\theta^+) + 12tU \frac{1}{2\pi i} \int \mu (w_\theta^- + w_\theta^+)\\
     = & C_{w_\theta} (L-12tzU) \mu + C_{\omega'_{\theta}} \mu\\
    & + 12it\frac{1}{2\pi i} \text{ad} \sigma \int \diamond \mu (w_\theta^- + w_\theta^+) + 12tU \frac{1}{2\pi i} \int \mu (w_\theta^- + w_\theta^+).
\end{align*}
We then deduce that 
\begin{align}
\label{L-12tzU}
     (L - 12tzU)\mu &= \left(1-C_{w_\theta}\right)^{-1}  \left[ C_{\omega'_{\theta}}\mu +  12it\frac{1}{2\pi i} \text{ad} \sigma \int \diamond \mu (w_\theta^- + w_\theta^+) \mu\right.\\
     \nonumber
     &\left. \quad + 12tU \frac{1}{2\pi i} \int \mu (w_\theta^- + w_\theta^+)\right]\\
     \nonumber
   &=\left(1-C_{w_\theta}\right)^{-1} C_{\omega'_{\theta}}\mu +\left( 12it\frac{1}{2\pi i} \text{ad} \sigma \int \diamond \mu (w_\theta^- + w_\theta^+)  + 12tU \frac{1}{2\pi i} \int \mu (w_\theta^- + w_\theta^+) \right)\mu.
\end{align}
Moreover, it is easy to see that
\begin{align}
     (\partial_z-\tilde{L})m_- &=(\partial_z-\tilde{L})\mu v_\theta^-\\
      \nonumber
                 &=((\partial_z-\tilde{L})\mu) v_\theta^-+\mu (v^-)'_\theta ,\\
      (\partial_z-\tilde{L})m_-^{-1} &=(\partial_z-\tilde{L})\left(v_\theta^-\right)^{-1} \mu^{-1}\\
      \nonumber
                 &=\left( v_\theta^-\right)^{-1}((\partial_z-\tilde{L})\mu^{-1}) +\left[ (v^-)'_\theta \right]^{-1} \mu^{-1}.          
\end{align}
We then proceed to decompose $\partial_z F$ into:
\begin{align}
\label{eq: parzF}
    \partial_z F &=c_1 \int e^{i \theta \mathrm{ad} \sigma}(\partial_z - \tilde{L})m_-^{-1} G m_-dy+c_2\int e^{i \theta \mathrm{ad} \sigma}m_-^{-1}G (\partial_z - \tilde{L})m_- dy - c_3\int e^{i \theta \mathrm{ad} \sigma}m_-^{-1} \tilde{L}G m_- dy \\
    \nonumber
    &=  \partial_z F_1+  \partial_z F_2+ \partial_z F_3.
\end{align}
We first consider $\partial_z F_2$. For $ \partial_z F_2$ we further decompose the expression $ \partial_z F_2$ into:
\begin{align}
    \partial_z F_2 &=\int e^{i \theta \mathrm{ad} \sigma}m_-^{-1} G (\partial_z - \tilde{L})m_- dy\\
    \nonumber
    &=\int e^{i \theta \mathrm{ad} \sigma}m_-^{-1} G  \left((\partial_z - \tilde{L}) \mu \right)v_\theta^- dy + \int e^{-i \theta \mathrm{ad} \sigma}m_-^{-1} G  \mu (v^-)'_\theta dy \\
    \nonumber
    &=\int e^{i \theta \mathrm{ad} \sigma} m_-^{-1} G\left[\left(1-C_{w_\theta}\right)^{-1} C_{\omega'_{\theta}}\mu\right] v_\theta^-dy+ \frac{6t}{\pi } \int e^{-i \theta \mathrm{ad} \sigma} G\left(\text{ad} \sigma \int \diamond \mu (w_\theta^- + w_\theta^+)\right)\mu v_\theta^- dy \\
    \nonumber
    &\quad  + \frac{6t}{\pi } \int e^{i \theta \mathrm{ad} \sigma} (m_-^{-1} -I)G\left(\text{ad} \sigma \int \diamond \mu (w_\theta^- + w_\theta^+)\right)\mu v_\theta^- dy  \\
    \nonumber
    &\quad + \frac{6t}{\pi i } \int e^{i \theta \mathrm{ad} \sigma} (m_-^{-1} -I)G U \left(\text{ad} \sigma \int \diamond \mu (w_\theta^- + w_\theta^+)\right)\mu v_\theta^- dy\\
    \nonumber
    &\quad  + \frac{6t}{\pi i } \int e^{i \theta \mathrm{ad} \sigma} G U \left(\text{ad} \sigma \int \diamond \mu (w_\theta^- + w_\theta^+)\right)\mu v_\theta^- dy +  \int e^{i \theta \mathrm{ad} \sigma}m_-^{-1} G  \mu (v^-)'_\theta dy \\
    \nonumber
    &= \partial_z F_{21}+ \partial_z F_{22}+ \partial_z F_{23}+ \partial_z F_{24}+ \partial_z F_{25} + \partial_z F_{26}.
\end{align}
In this case,
\begin{align*}
    \norm{ \partial_z F_{21}}{L^2(dz)} &\leq c \norm{\left(1-C_{w_\theta}\right)^{-1} C_{\omega'_{\theta}}\mu}{L^2(dz)} \norm{u^\ell}{L^\infty(dx)}\norm{m_-}{L^\infty(dz)}\norm{r}{L^\infty(dz)}\\
    &\leq c t^{-\ell/3},\\
\norm{ \partial_z F_{22}}{L^2(dz)} &\leq c t \norm{  \int \diamond \mu (w_\theta^- + w_\theta^+)}{L^2(dx)} \norm{\mu}{L^\infty(dz)} \norm{r}{L^\infty(dz)}\norm{u^{\ell}}{L^\infty(dx)}\\
    &\leq c t^{-(\ell-3)/3},\\
 \norm{ \partial_z F_{23}}{L^2(dz)} &\leq  c t \norm{  \int \diamond \mu (w_\theta^- + w_\theta^+)}{L^\infty(dz)} \norm{m_- -I}{L^2(dz)} \norm{\mu}{L^\infty(dz)}\norm{r}{L^\infty(dz)}\norm{u^{\ell}}{L^\infty(dx)}\\
    &\leq c t^{-(\ell-3)/3},\\
\norm{ \partial_z F_{24}}{L^2(dz)} &\leq  c t \norm{  \int \diamond \mu (w_\theta^- + w_\theta^+)}{L^\infty(dz)} \norm{m_- -I}{L^2(dz)}\norm{\mu}{L^\infty(dz)}\norm{r}{L^\infty(dz)}\norm{u^{\ell+1}}{L^\infty(dx)}\\
    &\leq c t^{-(\ell-2)/3},\\
\norm{ \partial_z F_{25}}{L^2(dz)} &\leq c t \norm{  \int \diamond \mu (w_\theta^- + w_\theta^+)}{L^2(dx)} \norm{\mu}{L^\infty(dz)}\norm{r}{L^\infty(dz)}\norm{u^{\ell+1}}{L^\infty(dx)}\\
    &\leq c t^{-(\ell-2)/3},\\
\norm{ \partial_z F_{26}}{L^2(dz)} &\leq c  \norm{  \int \diamond \mu (w_\theta^- + w_\theta^+)}{L^\infty(dx)} \norm{\mu}{L^\infty(dz)}\norm{r}{H^1(dz)}\norm{u^{\ell}}{L^\infty(dx)}\\
    &\leq c t^{-\ell/3}.
\end{align*}
The estimate of $\partial_z F_1$ is similar to that of $\partial_z F_2$. For the estimate of $\partial_z F_3$, we first introduce from \cite[p.228]{DZ-2} the following expression:
\begin{align*}
\tilde{L} U & =\int\left((\tilde{L} \mu) w_\theta+\mu \tilde{L} w_\theta\right)=\int\left((\tilde{L} \mu) w_\theta+\mu\left(\partial_z e^{i \theta \mathrm{ad} \sigma}\right) w\right)=\int\left(\left(\left(-\partial_z+\tilde{L}\right) \mu\right) w_\theta-\mu w_\theta^{\prime}\right) \\
& =-\int\left(((L-12z t U) \mu) w_\theta-(\mu-I) w_\theta^{\prime}\right)+\int w_\theta^{\prime} \\
& \equiv L_U+L_U^{\prime}.
\end{align*}
Notice that by \eqref{L-12tzU}, we have for $L_U$:
\begin{align}
    L_U&=-\int \left[\left(1-C_{w_\theta}\right)^{-1} C_{\omega'_{\theta}}\mu \right]w_\theta dz\\
    \nonumber
    &
   \quad  + \int\left( 12it\frac{1}{2\pi i} \text{ad} \sigma \int \diamond \mu (w_\theta^- + w_\theta^+)  + 12tU \frac{1}{2\pi i} \int \mu (w_\theta^- + w_\theta^+) \right)\mu w_\theta dz\\
    \nonumber
    & \quad + \int (\mu-I)w'_\theta dz\\
    \nonumber
    &=L_{U,1} + L_{U,2} +L_{U,3}.
\end{align}
It is then easy to deduce that 
\begin{align}
    \norm{L_{U,1} }{L^\infty(dx)} &\leq c \norm{\left(1-C_{w_\theta}\right)^{-1} C_{\omega'_{\theta}}\mu}{L^2(dz)}\norm{w_\theta}{L^2(dz)}\leq c \norm{r}{H^{1}(dz)}\norm{\mu}{L^\infty(dz)} \norm{r}{L^2(dz)},\\
\norm{L_{U,2} }{L^\infty(dx)}& \leq ct \norm{ \int \diamond \mu (w_\theta^- + w_\theta^+)}{L^\infty(dx)} \norm{\mu}{L^\infty(dz)} \norm{r}{L^{2,1}},\\
 \norm{L_{U,3} }{L^\infty(dx)} &\leq c \norm{\mu-I}{L^2(dz)}\norm{r}{H^1(dz)},\\
\norm{L'_U }{L^2(dx)} &\leq c \norm{r}{H^1}.
\end{align}
Notice that 
\begin{equation}
    \widetilde{L}G= -\left(i x \operatorname{ad} \sigma+12 t z\partial_x\right)\twomat{0}{u}{u}{0}u^{\ell-1}=\widetilde{L}U u^{\ell-1}
\end{equation}
so we can decompose $\partial_z F_3$ into:
\begin{equation}
    \int e^{-i \theta \mathrm{ad} \sigma}m_-^{-1} \tilde{L}G m_- dy=\int e^{-i \theta \mathrm{ad} \sigma}m_-^{-1} u^{\ell-1}\left(L_{U,1} + L_{U,2} +L_{U,3}+L'_U \right) m_- dy
\end{equation}
and deduce that 
\begin{equation}
    \norm{\partial_z F_3}{L^2(dz) }\leq c t^{-(\ell-4)/3}\norm{m}{L^\infty(dz)}.
\end{equation}
We combine all the constants $c$ in the proof into $C^{1,0}_F(\rho.\eta)$ and conclude the proof.
\end{proof}
\subsection{Estimating the difference}
\begin{remark}
  In this subsection, to avoid over-complicating things, when obtaining necessary estimates,  we will only deal with certain terms that lead to the lowest order of time decay. And it will become clear that by doing so we can cover all the necessary steps in the estimation of the remaining terms with higher order of decay.
\end{remark}
\begin{lemma}
\label{le:DF}
    \begin{equation}
        \|\Delta F\|_{L^2(dz)} \leq \frac{C^{0,0}_{\Delta F}(\rho, \eta)}{(1+t)^{(\ell-3)/3}} \|\Delta r\|_{H^{1,2}_{\scriptscriptstyle 1/2}}
    \end{equation}
    where the constant $C^{0,0}_{\Delta F}(\rho, \eta)$ is uniform in $x$ and $t$ and are monotonic in $\rho$ and $\eta$.
\end{lemma}
\begin{proof}
    Decomposing $F$ into four parts as given by \eqref{eq:decomF}, we estimate $I_i:= \Delta F^{(i)}$, $i=1,...4$ term by term:
    \begin{eqnarray*}
        \norm{I_1}{L^2 (dz)} = c\|\Delta G\|_{L^2(dx)} &\leq&c\|u\|_{L^\infty(dx)}^{\ell-3} \norm{u}{L^2(dx)}^2\|\Delta u\|_{L^\infty(dx)}\\
        &\leq& c t^{-(\ell-3)/3} \norm{\Delta r}{H^{1,2}_{\scriptscriptstyle 1/2}}
    \end{eqnarray*}
    where we use the estimate of $\|\Delta u\|_{L^\infty(dx)}$ and $\|\Delta m\|_{L^2(dz)}$ in \cite[Lemma 4.16 and Lemma 5.12]{DZ-2}.
    \begin{eqnarray*}
       \norm{I_2}{L^2 (dz)} & = & c_1\|\Delta G\|_{L^1(dx)} \|m_--I\|_{L^2(dz)\otimes L^\infty(dx)}+ c_2\|G\|_{L^1(dx)} \|\Delta m_-\|_{L^2(dz)\otimes L^\infty(dx)}\\
        &\leq& c_1 \|u\|_{L^\infty(dx)}^{\ell-3} \norm{u}{L^2(dx)}^2 \|\Delta u\|_{L^\infty(dx)}\|m_--I\|_{L^2(dz)\otimes L^\infty(dx)}+ c_2\|u\|_{L^\infty(dx)}^{\ell-2} \norm{u}{L^2(dx)}^2\|\Delta m_-\|_{L^2(dz)\otimes L^\infty(dx)}\\
       &\leq& c t^{-(\ell-3)/3} \norm{\Delta r}{H^{1,2}_{\scriptscriptstyle 1/2}}.
    \end{eqnarray*}

    The estimate of
    \[ \norm{I_3}{L^2 (dz)} = c_1\|\Delta G\|_{L^1(dx)} \|m_-^{-1}-I\|_{L^2(dz)\otimes L^\infty(dx)}+ c_2\|G\|_{L^1(dx)} \|\Delta m_-^{-1}\|_{L^2(dz)\otimes L^\infty(dx)}\]
    is similar to $I_2$.
    
    \begin{eqnarray*}
        \norm{I_4}{L^2 (dz)} & = & c_1\|\Delta m_-^{-1}\|_{L^4(dz)\otimes L^\infty(dx)}\|G\|_{L^1(dx)} \|m_--I\|_{L^4(dz)\otimes L^\infty(dx)}\\
        &&+ c_2\|m_-^{-1}-I\|_{L^4(dz)\otimes L^\infty(dx)}\|\Delta G\|_{L^1(dx)} \|m_--I\|_{L^4(dz)\otimes L^\infty(dx)}\\
        &&+ c_3\|m_-^{-1}-I\|_{L^4(dz)\otimes L^\infty(dx)}\|G\|_{L^1(dx)} \|\Delta m_-^{-1}\|_{L^4(dz)\otimes L^\infty(dx)}\\
        &\leq& \left(c_1 t^{-(\ell-2)/3}+ c_2 t^{-(\ell-3)/3}+ c_3 t^{-(\ell-2)/3}\right)\norm{\Delta r}{H^{1,2}_{\scriptscriptstyle 1/2}}\\
        &\leq & c t^{-(\ell-3)/3}\norm{\Delta r}{H^{1,2}_{\scriptscriptstyle 1/2}}.
    \end{eqnarray*}
We combine all the constants $c$ in the proof into $C^{0,0}_{\Delta F}(\rho.\eta)$ and conclude the proof.
\end{proof}

\begin{lemma}
\label{le:DzF}
    \begin{equation}
        \|\Delta \diamond^2 F\|_{L^2(dz)}  \leq \frac{C^{0,2}_{\Delta  F}(\rho, \eta)}{(1+t)^{(\ell-3)/3}}\|\Delta r\|_{H^{1,2}_{\scriptscriptstyle 1/2}}
    \end{equation}
    where the constant $C^{0,2}_{\Delta F}(\rho, \eta)$ is uniform in $x$ and $t$ and are monotonic in $\rho$ and $\eta$.
\end{lemma}

\begin{proof}
   As is given by \eqref{eq:decomF}, the first part of  $\|\Delta \diamond^2 F^{(1)}\|_{L^2(dz)}$ is equal to $\|\Delta \partial^2_{xx} G\|_{L^2(dx)}$ according to the Plancherel's theorem. 

    \begin{eqnarray*}
        \|\Delta \diamond^2 F^{(1)}\|_{L^2(dz)} &=& c\|\Delta \partial^2_{xx}G\|_{L^2(dx)}\\
        \nonumber
        &\leq& c_1  \|u\|_{L^\infty(dx)}^{\ell-2}\|\partial^2_{xx} u\|_{L^2(dx)}\|\Delta u\|_{L^\infty(dx)}+c_2 \|u\|_{L^\infty(dx)}^{\ell-2}\|\Delta\partial^2_{xx} u\|_{L^\infty(dx)}\|u\|_{L^2(dx)}\\
        \nonumber
        &&+c_3\|u_x\|_{L^2(dx)} \|u\|_{L^\infty(dx)}^{\ell-3}\|\Delta u\|_{L^\infty(dx)}+c_4\|\Delta u_x\|_{L^\infty(dx)} \|u\|_{L^\infty(dx)}^{\ell-2}\|u_x\|_{L^2(dx)}.
       \end{eqnarray*}
Recall that 
\begin{equation}
    U(x,t)=\frac{1}{2\pi i} \ad\sigma \int_{\mathbb{R}} \mu(x, z)\left(w_\theta^{+}(z)+w_\theta^{-}(z)\right) d z.
\end{equation}
And we can decompose $\Delta U_{x}$ as 
\begin{align}
  \Delta U_x &=   \frac{1}{2\pi i} \ad\sigma \int_{\mathbb{R}} \Delta \left[ \mu(x, z)_x\left(w_\theta^{+}(z)+w_\theta^{-}(z)\right) \right]d z \\
  \nonumber
   &\quad + \frac{1}{2\pi i} \ad\sigma \int_{\mathbb{R}} \diamond \mu(x, z) \Delta \left(w_\theta^{+}(z)+w_\theta^{-}(z)\right) d z .
\end{align}
We further calculate that 
\begin{align}
    U_{xx} &=\frac{1}{2\pi i} \ad\sigma \int_{\mathbb{R}} \mu_{xx}(x, z)\left(w_\theta^{+}(z)+w_\theta^{-}(z)\right) d z \\
   \nonumber
   &\quad +\frac{1}{2\pi i} \ad\sigma \int_{\mathbb{R}}  \mu(x, z)\diamond^2\left(w_\theta^{+}(z)+w_\theta^{-}(z)\right) d z \\
   \nonumber
   &\quad +\frac{1}{2\pi i} \ad\sigma \int_{\mathbb{R}}  \mu(x, z)_x \diamond\left(w_\theta^{+}(z)+w_\theta^{-}(z)\right) d z .
   \end{align}
 Moreover we have that
   \begin{align}
       \mu_{xx}=-iz[\sigma, \mu]_x+U_x\mu+ U\mu_x
   \end{align}
   and
 \begin{align}
     \Delta \mu_x=-iz [\sigma, \Delta\mu] +\Delta U\mu+U\Delta\mu,
 \end{align}
 and
 \begin{align}
    \Delta \mu_{xx} =&-iz \left[\sigma, -iz [\sigma, \Delta\mu] +\Delta U\mu+U\Delta\mu \right] +\Delta U_x \mu\\
    \nonumber
    &+ U \left( -iz [\sigma, \Delta\mu] +\Delta U\mu+U\Delta\mu \right).
 \end{align}
 
 So we have the following bound:
 \begin{align}
 \norm{ \Delta U_{x}}{L^\infty(dx)},\quad   \norm{ \Delta U_{xx}}{L^\infty(dx)} \leq c \norm{\Delta r}{H^{1,2}}.
 \end{align}
 So we conclude that 
 \begin{equation}
     \|\Delta \diamond^2 F^{(1)}\|_{L^2(dz)} \leq c t^{-(\ell-2)/3} \norm{\Delta r}{H^{1,2}}.
 \end{equation}
 We then analyze $\|\Delta \diamond^2 F^{(3)}\|_{L^2(dz)}$ since the estimate of $\|\Delta \diamond^2 F^{(2)}\|_{L^2(dz)}$ is similar. Decomposing it as we have done in the proof of Lemma \ref{le:zF2}:  we have $\Delta z^2F^{(3)}= \Delta I_1+ \Delta I_2+ \Delta I_3$ with
\begin{align}
    \Delta I_1 &= \Delta \int e^{-2i \theta}  G \left[C^-\diamond^2 \mu (w_++w_-)\right]dy\\
    \nonumber
    &=\int e^{-2i \theta}  \Delta G \left[C^-\diamond^2 \mu (w_++w_-)\right]dy + \int e^{-2i \theta}   G \Delta\left[C^-\diamond^2 \mu (w_++w_-)\right]dy\\
    \Delta I_2 &=  \Delta \int e^{-2i \theta} ( GU)_x dy, \\
     \Delta I_3 &= \Delta \int e^{-2i \theta} G\left(\frac{1}{2 \pi i} \int\diamond \mu (w_++w_-) \right)dy\\
     \nonumber
      &=  \int e^{-2i \theta} \Delta G\left(\frac{1}{2 \pi i} \int\diamond \mu (w_++w_-) \right)dy +  \int e^{-2i \theta} G \Delta\left(\frac{1}{2 \pi i} \int\diamond \mu (w_++w_-) \right)dy.
\end{align}
For $\Delta I_1$ we  consider 
\begin{align}
    \norm{\Delta C^-\diamond^2 \mu (w_++w_-)}{L^2(dz)} &\leq \norm{ C^- \Delta\mu \left[\diamond^2 (w_++w_-) \right]}{L^2(dz)} +\norm{ C^- \mu \Delta\left[\diamond^2 (w_++w_-) \right]}{L^2(dz)} \\
    \nonumber
   & \leq \norm{\Delta\mu \left[\diamond^2 (w_++w_-) \right]}{L^2(dz)}+\norm{ \mu \Delta\left[\diamond^2 (w_++w_-) \right]}{L^2(dz)}\\
   \nonumber
   &\leq |z_0|^2 C_{\Delta \mu}^2 (\eta)\norm{\Delta r}{H^{1,2}_{\scriptscriptstyle 1/2}} \norm{r}{L^{2,2}(dz)}+C_{ \mu}^2 (\eta)\norm{r}{H^{1,2}_{\scriptscriptstyle 1/2}} \norm{\Delta r}{L^{2,2}(dz)}.
\end{align}
   
Thus,
\begin{align}
    \norm{\int e^{-2i \theta}   G \Delta\left[C^-\diamond^2 \mu (w_++w_-)\right]dy}{L^2(dz)} &\leq C_{\Delta \mu}^2 (\eta)\norm{\Delta r}{H^{1,2}_{\scriptscriptstyle 1/2}} \norm{r}{L^{2,2}(dz)} \int \left\vert z_0^2 G\right\vert dy\\
    \nonumber
    &\quad +C_{ \mu}^2 (\eta)\norm{r}{H^{1,2}_{\scriptscriptstyle 1/2}} \norm{\Delta r}{L^{2,2}(dz)}  \int \left\vert G\right\vert dy\\
    \nonumber
    &\leq c t^{-(\ell-2)/3}\norm{\Delta r}{H^{1,2}_{\scriptscriptstyle 1/2}}.
\end{align}
And we conclude that 
\begin{equation}
    \norm{\Delta I_1}{L^2(dz)}\leq c t^{-(\ell-3)/3}\norm{\Delta r}{H^{1,2}_{\scriptscriptstyle 1/2}}.
\end{equation}
Similarly,
\begin{equation}
    \norm{\Delta I_2}{L^2(dz)}\leq c t^{-(\ell-2)/3}\norm{\Delta r}{H^{1,2}_{\scriptscriptstyle 1/2}}.
\end{equation}
Finally,
\begin{align}
  \norm{\Delta I_3}{L^2(dz)}&\leq  \norm{\int e^{-2i \theta} \Delta G\left(\frac{1}{2 \pi i} \int\diamond \mu (w_++w_-) \right)dy}{L^2(dz)}\\
  \nonumber 
  &\quad +\norm{\int e^{-2i \theta}  G \Delta\left(\frac{1}{2 \pi i} \int\diamond \mu (w_++w_-) \right)dy}{L^2(dz)}\\
  \nonumber
  &\leq c t^{-(\ell-2)/3}\norm{\Delta r}{H^{1,2}_{\scriptscriptstyle 1/2}}.
\end{align}
Finally for $\Delta\diamond^2 F^{(4)}$,  we only estimate
\begin{align}
     \norm{\int e^{-2i\theta} C^-\left[\diamond (\Delta\mu) (w_++w_-)\right]G C^-\left[\diamond \mu (w_++w_-)\right] }{L^2(dz)}&\leq \norm{C^-\left[\diamond (\Delta\mu) (w_++w_-)\right]}{L^4(dz)}^2\\
     \nonumber
    &\quad\times \norm{C^-\left[\diamond (\mu) (w_++w_-)\right]}{L^4(dz)}^2\norm{G}{L^1(dx)}\\
    \nonumber
    &\leq c \norm{(\Delta\mu) \diamond r}{L^2(dz)}^2\norm{\mu \diamond r}{L^2(dz)}^2\norm{z_0^2G}{L^1(dx)}
\end{align}
 which follows from Lemma \ref{lm:deltaM}, Lemma \ref{lm:deltaM'}-Lemma \ref{lm:deltaM'''} and conclude that
\begin{equation}
    \norm{\Delta\diamond^2 F^{(4)}}{L^2(dz)}\leq c t^{-(\ell-2)/3}\norm{\Delta r}{H^{1,2}_{\scriptscriptstyle 1/2}}.
\end{equation}
We combine all the constants $c$ in the proof into $C^{0,2}_{\Delta F}(\rho.\eta)$ and conclude the proof.
\end{proof}
\begin{lemma}
\label{le:dDF}
    \begin{equation}
        \|\Delta \partial_z F\|_{L^2(dz)} \leq \frac{C^{1,0}_{\Delta F}(\rho, \eta)}{(1+t)^{\frac{l}{2}+\frac{1}{2p}-\frac{3}{4}}}\|\Delta r\|_{H^{1,2}_{\scriptscriptstyle 1/2}}.
    \end{equation}
    where the constant $C^{1,0}_{\Delta F}(\rho, \eta)$ is uniform in $x$ and $t$ and are monotonic in $\rho$ and $\eta$.
\end{lemma}

\begin{proof}
    As is given by \eqref{eq: parzF}, we again obtain the estimates term by term. We only need to calculate one of $\Delta \partial_z F_1$ and $\Delta \partial_z F_2$. For $\Delta \partial_z F_2$, we use the decomposition given in Lemma \ref{le:dF}  and first consider
\begin{align}
    \Delta \partial_z F_{21}&= \Delta\int e^{i \theta \mathrm{ad} \sigma} m_-^{-1} G\left[\left(1-C_{w_\theta}\right)^{-1} C_{\omega'_{\theta}}\mu\right] v_\theta^-dy\\
    \nonumber
    &= \int e^{i \theta \mathrm{ad} \sigma}( \Delta m_-^{-1} )G\left[\left(1-C_{w_\theta}\right)^{-1} C_{\omega'_{\theta}}\mu\right] v_\theta^-dy\\
\nonumber
&\quad +\int e^{i \theta \mathrm{ad} \sigma}m_-^{-1} ( \Delta  G)\left[\left(1-C_{w_\theta}\right)^{-1} C_{\omega'_{\theta}}\mu\right] v_\theta^-dy\\
\nonumber
&\quad +\int e^{i \theta \mathrm{ad} \sigma}m_-^{-1}  G \left( \Delta \left[\left(1-C_{w_\theta}\right)^{-1} C_{\omega'_{\theta}}\mu\right] \right)v_\theta^-dy\\
\nonumber
&\quad +\int e^{i \theta \mathrm{ad} \sigma}m_-^{-1} G\left[\left(1-C_{w_\theta}\right)^{-1} C_{\omega'_{\theta}}\mu\right] (\Delta v_\theta^-)dy.
\end{align}
We only consider the third term. We first note that by \eqref{est:dm}, for $x<0$,
\begin{align*}
    \norm{\left( \Delta \left[\left(1-C_{w_\theta}\right)^{-1} C_{\omega'_{\theta}}\mu\right] \right)}{L^2(dz)}\leq c  \dfrac{z_0^2 F(x,t,\eta)\norm{\Delta r}{H^{1,2}_{\scriptscriptstyle 1/2}} }{1-(1-\rho)^{-1}\norm{  r'}{L^{2}}}.
\end{align*}
So we obtain that 
\begin{equation}
    \norm{\int e^{i \theta \mathrm{ad} \sigma}m_-^{-1}  G \left( \Delta \left[\left(1-C_{w_\theta}\right)^{-1} C_{\omega'_{\theta}}\mu\right] \right)v_\theta^-dy}{L^2(dz)}\leq c t^{-(\ell-5)/3}\norm{\Delta r}{H^{1,2}_{\scriptscriptstyle 1/2}}
\end{equation}
and conclude that
\begin{equation}
    \norm{\Delta \partial_z F_{21}}{L^2(dz)}\leq c t^{-(\ell-5)/3}\norm{\Delta r}{H^{1,2}_{\scriptscriptstyle 1/2}}.
\end{equation}
For $\Delta\partial_z F_{22} $, again by \eqref{est:dm}, we obtain
\begin{align}
    \norm{\frac{6t}{\pi } \int e^{-i \theta \mathrm{ad} \sigma} G\left(\text{ad} \sigma \int \diamond \mu (w_\theta^- + w_\theta^+)\right)(\Delta\mu) v_\theta^- dy }{L^2(dz)}\leq ct^{-{(\ell-4)/3}} \norm{\Delta r}{H^{1,2}_{1/2}} ,
\end{align}
and
\begin{align}
    \norm{\frac{6t}{\pi } \int e^{-i \theta \mathrm{ad} \sigma} G\left(\text{ad} \sigma \int (\Delta\mu)\diamond(w_\theta^- + w_\theta^+)\right) \mu v_\theta^- dy }{L^2(dz)}\leq  ct^{-{(\ell-4)/3}} \norm{\Delta r}{H^{1,2}_{1/2}}  .
\end{align}
Since the rest of the terms can be estimated similarly,  we conclude that 
\begin{equation}
    \norm{\Delta \partial_z F_{22}}{L^2(dz)}\leq c t^{-(\ell-5)/3}\norm{\Delta r}{H^{1,2}_{\scriptscriptstyle 1/2}}.
\end{equation}
Furthermore we have
\begin{equation}
    \norm{\Delta \partial_z F_{23}}{L^2(dz)},  \norm{\Delta \partial_z F_{24}}{L^2(dz)}, \norm{\Delta \partial_z F_{25}}{L^2(dz)}, \norm{\Delta \partial_z F_{26}}{L^2(dz)}\leq c t^{-(\ell-5)/3}\norm{\Delta r}{H^{1,2}_{\scriptscriptstyle 1/2}}
\end{equation}
and this concludes the estimate of $\Delta \partial_z F_{2}$. For the estimate of $\Delta \partial_z F_3$, we recall the definition of $\partial_z F_3$ given in Lemma \ref{le:DF} and only consider the following term:
\begin{align}
    \Delta\int e^{-i \theta \mathrm{ad} \sigma}m_-^{-1} u^{\ell-1}\left(L_{U,1} \right) m_- dy&=\int e^{-i \theta \mathrm{ad} \sigma}(\Delta m_-^{-1}) u^{\ell-1}\left(L_{U,1} \right) m_- dy+\int e^{-i \theta \mathrm{ad} \sigma}m_-^{-1} (\Delta u^{\ell-1})\left(L_{U,1} \right) m_- dy\\
    \nonumber
    &\quad + \int e^{-i \theta \mathrm{ad} \sigma}m_-^{-1} u^{\ell-1}\left( \Delta L_{U,1} \right) m_- dy+ \int e^{-i \theta \mathrm{ad} \sigma}m_-^{-1} u^{\ell-1}\left(L_{U,1} \right)(\Delta m_-) dy.
\end{align}
Notice that
\begin{align}
    \norm{\Delta L_{U,1}}{L^\infty (dx)}&\leq \norm{\int \left[ \left(1-C_{w_\theta}\right)^{-1} C_{\omega'_{\theta}}\mu \right](\Delta w_\theta) dz}{L^\infty(dz)}\\
    \nonumber
    &\quad + \norm{\int \left[\Delta \left(1-C_{w_\theta}\right)^{-1} C_{\omega'_{\theta}}\mu \right]w_\theta dz}{L^\infty(dz)}.
    \end{align}
    And by \eqref{est:dm}
  \begin{align}
      \norm{\int \left[ \left(1-C_{w_\theta}\right)^{-1} C_{\omega'_{\theta}}\Delta \mu \right]w_\theta dz}{L^\infty(dz)} &\leq \norm{\left(1-C_{w_\theta}\right)^{-1}}{L^2(dz)}   \dfrac{z_0^2 F(x,t,\eta)\norm{\Delta r}{H^{1,2}} }{1-(1-\rho)^{-1}\norm{  r'}{L^{2}}}\norm{r}{L^2(dz)}
  \end{align}
    we can conclude that 
    \begin{equation}
      \norm{ \Delta\int e^{-i \theta \mathrm{ad} \sigma}m_-^{-1} u^{\ell-1}\left(L_{U,1} \right) m_- dy}{L^2(dz)}  \leq c t^{-(\ell-6)/3}\norm{\Delta r}{H^{1,2}_{\scriptscriptstyle 1/2}}
    \end{equation}
    and eventually obtain 
    \begin{equation}
        \norm{\Delta \partial_z F_3}{L^2(dz)}=\norm{\int e^{-i \theta \mathrm{ad} \sigma}m_-^{-1} \tilde{L}G m_- dy}{L^2(dz)}\leq  c t^{-(\ell-6)/3}\norm{\Delta r}{H^{1,2}_{\scriptscriptstyle 1/2}}.
    \end{equation}
    We combine all the constants $c$ in the proof into $C^{1,0}_{\Delta F}(\rho.\eta)$ and conclude the proof.
    \end{proof}

\section{$L^\infty$ estimates through $\overline{\partial}$-nonlinear steepest descent}
\label{sec:dbar}
In this section, we will make use of the $\overline{\partial}$-steepest descent method to obtain $L^\infty$ bound on $m_\pm$ and $\mu$ which are necessary for the proof of Theorem \ref{thm:r}. We refer the reader to \cite[lemma 5.17, 5.23]{DZ-2} where $L^\infty$-norm plays a key role in obtaining the important estimates \cite[(6.5)-(6.6)]{DZ-2}. We then prove a new type of \textit{a priori} estimate  which is an alternative version of \cite[Theorem 1.7]{DZ-1} in the context of mKdV.
We first rewrite \eqref{eq:r} as follows:
\begin{align}
    r(t)(z)&=r_0(z)+\varepsilon \int_0^t d s e^{-8i z^3 s} \int_{-\infty}^{+\infty} d y e^{-2i y z}\left(m_{-}^{-1}(z ; y,u(s)) G(u(y, s)) m_{-}(z ; y, u(s))\right)_{21}\\
    \nonumber
           &=r_0(z)+\varepsilon \int_0^t d s e^{-8i z^3 s} \int_{-\infty}^{-t^{1/3
           }} d y e^{-2i y z}\left(m_{-}^{-1}(z ; y,u(s)) G(u(y, s)) m_{-}(z ; y, u(s))\right)_{21}\\
           \nonumber
            & \quad + \varepsilon \int_0^t d s e^{-8i z^3 s} \int_{-t^{1/3}}^{t^{1/3}} d y e^{-2i y z}\left(m_{-}^{-1}(z ; y,u(s)) G(u(y, s)) m_{-}(z ; y, u(s))\right)_{21}\\
           \nonumber
           & \quad + \varepsilon \int_0^t d s e^{-8i z^3 s} \int_{t^{1/3}}^{+\infty} d y e^{-2i y z}\left(m_{-}^{-1}(z ; y,u(s)) G(u(y, s)) m_{-}(z ; y, u(s))\right)_{21}\\
           \nonumber
           &=r_0+ \varepsilon \textrm{R}_1+ \varepsilon\textrm{R}_2+ \varepsilon\textrm{R}_3.
\end{align}
\begin{remark}
Assume that $t\gg 1$, we remark that
\begin{enumerate}
    \item For  $\textrm{R}_1$, the limits of integration corresponds to Region I and Region II in Figure \ref{fig:regions}. And more specifically, we have 
$$z_0=\sqrt{\frac{-x}{12 t}}>ct^{-1/3}, \quad z_0 t>c t^{2/3}.$$

\item For $\textrm{R}_2$, the limits of integration corresponds to Region III in Figure \ref{fig:regions} with
$$|x|/t^{1/3}\leq c.$$
\item For $\textrm{R}_3$, the limits of integration corresponds to Region IV-V in Figure \ref{fig:regions} with
$$\tau=\left(\frac{x}{12 t^{1 / 3}}\right)^{3 / 2}>c>0.$$
\end{enumerate}
From the conclusions of Theorem \ref{thm:main1} we can have the following uniform time decay of $u(x,t)$:
\begin{equation}
    u(x,t)\leq c/(1+t)^{1/3}
\end{equation}
where the constant c is independent of $x$ and $t$.
\end{remark}
\begin{remark}
    Throughout the rest of this section, the subscript $i=1,2$ denotes two reflection coefficients 
    \begin{equation}
        r_1, r_2\in H^{1,2}_{\scriptscriptstyle 1/2}(\bbR):= H^{1,2}(\bbR)\cap \lbrace f: \norm{f}{H^1}<1/2 \rbrace
    \end{equation}
    and also terms that consist of $r_i$. $\Delta$ is the difference operator: $\Delta f=f_1-f_2$.
\end{remark}
\subsection{the analysis of $\textrm{R}_1$}
Recall that in \cite[Section 2-Section 6]{CL}, a series of transformations are made to the solution $m(z; x, t)$ to Problem \ref{prob:mKdV.RH0}
\begin{equation}
\label{m:trans}
    m(z ; x, t)=m^{(3)}(z ; x, t) m^{\mathrm{LC}}\left(z ; z_0\right) \mathcal{R}^{(2)}(z)^{-1} \delta(z)^{\sigma_3}
\end{equation}
where $ m^{(3)}$ satisfies the following $\dbar$-integral equation:
\begin{equation}
\label{eq:m3}
    m^{(3)}(z ; x, t) =I+\frac{1}{\pi} \int_{\mathbb{C}} \frac{\bar{\partial}  m^{(3)}(s)}{s-z} d A(s)=I+\frac{1}{\pi} \int_{\mathbb{C}} \frac{ m^{(3)}(s) W(s)}{s-z} d A(s).
\end{equation}
\subsubsection{$\norm{m}{L^\infty}$ estimate}
We refer the reader to \cite[Problem 2.1 ]{CL}, \cite[Problem 4.1 ]{CL}, \cite[(3.3)-(3.10)]{CL} for the definitions of $\delta(z)$, $m^{\mathrm{LC}}$ and $ \mathcal{R}^{(2)}$ respectively.
Equation \eqref{eq:m3} can be rewritten as 
\begin{equation}
    (I-S)\left[ m^{(3)}(z)\right]=I
\end{equation}
where
\begin{equation}
\label{op:S}
    S[f]=\frac{1}{\pi} \int_{\mathbb{C}} \frac{f(s) W(s)}{s-z} d A(s).
\end{equation}
Here ( see \cite[(5.1)]{CL})
\begin{equation}
    W(z ; x, t)=m^{\mathrm{LC}}(z ; x, t) \bar{\partial} \mathcal{R}^{(2)}(z) m^{\mathrm{LC}}(z ; x, t)^{-1}.
\end{equation}
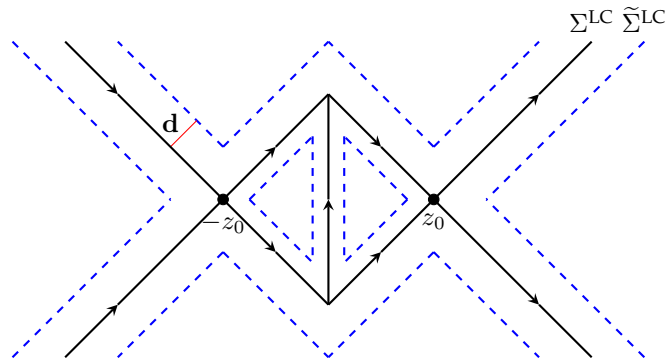
\begin{figure}[H]
\caption{$\Sigma^{\textrm{LC}}$}
\vskip 15pt
\begin{tikzpicture}[scale=0.7]
\draw [thick,dashed,blue] 					(-6, -3) -- (-3,0);	
\draw[thick, dashed, blue] 					(-6,3) -- (-3,0);

\draw [thick,dashed, blue] 					(6, 3) -- (3,0);	
\draw[thick, dashed, blue] 					(6,-3) -- (3,0);

\draw [thick,dashed, blue] 					(2, 1) -- ( 4,3);
\draw [thick,dashed, blue] 					(2, -1) -- ( 4,-3);

\draw [thick,dashed, blue] 					(-2, 1) -- ( -4,3);
\draw [thick,dashed, blue] 					(-2, -1) -- ( -4,-3);

\draw [thick,dashed, blue] 					(-2, 1) -- ( 0,3);
\draw [thick,dashed, blue] 					(-2, -1) -- ( 0,-3);

\draw [thick,dashed, blue] 					(2, 1) -- ( 0,3);
\draw [thick,dashed, blue] 					(2, -1) -- ( 0,-3);

\draw [thick,dashed, blue] 					(-1.5, 0) -- ( -0.3, 1.2);
\draw [thick,dashed, blue] 					(-1.5, 0) -- ( -0.3, -1.2);
\draw[thick, dashed, blue]                     (-0.3, 1.2)--(-0.3, -1.2);

\draw [thick,dashed, blue] 					(1.5, 0) -- ( 0.3, 1.2);
\draw [thick,dashed, blue] 					(1.5, 0) -- ( 0.3, -1.2);
\draw[thick, dashed, blue]                     (0.3, 1.2)--(0.3, -1.2);
\draw[red]	(-3,1) -- (-2.5, 1.5);		
\node[above] at (-3, 1.1) {$\mathbf{d}$};

\draw	[->,thick,>=stealth] 	(2,0) -- (4,2);								
\draw	[thick]	(5, 3) -- (4,2);
\draw  [->,thick,>=stealth] 	(-5,3) -- (-4,2);							
\draw [thick]	(-2,0) -- (-4,2);
\draw[->,thick,>=stealth]		(-5,-3) -- (-4,-2);							
\draw[thick]						(-4,-2) -- (-2,0);
\draw[thick,->,>=stealth]		(2,0) -- (4,-2);								
\draw[thick]						(4,-2) -- (5,-3);
\draw[thick,->,>=stealth]	(-2,0) -- (-1,1);								
\draw  [thick]    (0,2) -- (-1, 1);
\draw[thick,->,>=stealth]		(-2,0) -- (-1,-1);								
\draw[thick]						(-1,-1) -- (0, -2);
\draw	[thick,->,>=stealth]		(0,2) -- (1,1);								
\draw	[thick]  (2,0) -- (1, 1);

\draw[thick,->,>=stealth]		(0,-2) -- (1,-1);								
\draw[thick]						(1,-1) -- (2, -0);
\draw	[fill]							(-2,0)		circle[radius=0.1];	
\draw	[fill]							(2,0)		circle[radius=0.1];
\draw[->,thick,>=stealth] 		(0, -2) -- (0,0);
\draw[thick]			(0,0) -- (0,2);		

\node[below] at (-2,-0.1)			{$-z_0$};
\node[below] at (2,-0.1)			{$z_0$};

\node[above] at (5,3)			{$\Sigma^{\textrm{LC}}$};
\node[above] at (6,3)			{$\widetilde{\Sigma}^{\textrm{LC}}$};
\end{tikzpicture}
\label{fig:contour}
\end{figure}
We introduce the following lemma:
\begin{lemma}
\label{lm:mlc}
    \begin{align}
        \norm{m^{\mathrm{LC}}}{L^\infty}&\leq c \left(\norm{r}{H^1}+\norm{r}{L^{2,1}}\right),\\
        \label{dmlc}
          \norm{\Delta m^{\mathrm{LC}}}{L^\infty}&\leq c \left(\norm{\Delta r}{H^1}+\norm{\Delta r}{L^{2,1}}\right) \ln \left\vert \frac{z-z_0}{z+z_0}\right\vert.
    \end{align}
\end{lemma}
\begin{proof}
We will follow the ideas given by \cite{FIKN}. Recall that $\mlc$ is given by 
\begin{equation}
    m^{\mathrm{LC}}(z)= \begin{cases}E(z), & \left|z \pm z_0\right|>\rho, \\ E(z) m^{A^{\prime}}(z), & \left|z+z_0\right| \leq \rho, \\ E(z) m^{B^{\prime}}(z), & \left|z-z_0\right| \leq \rho.\end{cases}
\end{equation}
    We refer the reader to \cite[Problem 4.2-Problem 4.3, Problem 4.5]{CL} for the definition of $m^{A^{\prime}}(z), m^{B^{\prime}}(z)$ and $E(z)$, respectively. We also refer the reader to \cite[(3.3)-(3.10), (3.14)]{CL} for the jump matrices $v^{(2)}$ across $\Sigma^{\textrm{LC}}$ as shown in Figure \ref{fig:contour}.  To get the $L^\infty$ bound, fix a small constant \textbf{d} and suppose that $\inf _{z' \in \Sigma^{\textrm{LC}}}|z'-z|>\textbf{d}$, then
    \begin{equation}
        \left|m^{\mathrm{LC}}(z)-I\right| \leqslant \frac{\textbf{d}^{-1}}{2 \pi}\left(\left\|v^{(2)}-I\right\|_{L^1(dz)}+\left\|m^{\mathrm{LC}}(z)-I\right\|_{L^2(dz)}\left\|v^{(2)}-I\right\|_{L^2(dz)}\right).
    \end{equation}
    And it is also clear that for $\inf _{z' \in \Sigma^{\textrm{LC}}}|z'-z|>\textbf{d}$
  \begin{equation}
        \left|\Delta m^{\mathrm{LC}}(z)\right| \leqslant \frac{\textbf{d}^{-1}}{2 \pi}\left(\left\|\Delta v^{(2)}\right\|_{L^1}+\left\|\Delta m^{\mathrm{LC}}(z)\right\|_{L^2}\left\|v^{(2)}-I\right\|_{L^2}+\left\| m^{\mathrm{LC}}(z)-I\right\|_{L^2}\left\|\Delta v^{(2)}\right\|_{L^2}\right).
    \end{equation}  
 To get $L^\infty$ control for $z$ approaching $\SigmaLC$ we observe that the jump matrix $v^{(2)}$ on the contour $\SigmaLC$ (Figure \ref{fig:contour})
are locally analytic, so the contour  $\SigmaLC$ can be freely deformed to a new contour $\SigmaLCT$ , with no points of self-intersection, by a bounded invertible transformation
$m^{\textrm{LC}}(z) \mapsto \widetilde{m}^{\textrm{LC}}(z)$. More explicitly we have that
\begin{equation}
     \widetilde{m}^{\textrm{LC}}(z):=\mlc(z) \left( v^{(2)}(z)\right)^{-1}.
\end{equation}
The previous argument then goes through to show that $\mlcT(z)-I$ is bounded
on $\SigmaLC$ which then gives a similar bound on $\mlc(z)$ as the transformation itself is bounded. More explicitly, for $z\in \SigmaLC$
\begin{equation}
\left\vert  \mlc(z) \right\vert \leq   \left\vert \widetilde{m}^{\textrm{LC}}(z)\right\vert \left\vert v^{(2)}(z)\right\vert.
\end{equation}
\begin{equation}
\left\vert \Delta \mlc(z) \right\vert \leq   \left\vert \Delta \widetilde{m}^{\textrm{LC}}(z)\right\vert \left\vert \left(v_1^{(2)}(z)\right)^{-1}\right\vert +  \left\vert  \widetilde{m}^{\textrm{LC}}(z)\right\vert \left\vert \Delta \left(v^{(2)}(z)\right)^{-1}\right\vert.
\end{equation}
\end{proof}
We will restrict our analysis to region $\Omega_1$ given by \cite[Figure 3.1]{CL} and drop the subscript $1$.
\begin{align}
   \norm{ S(I)}{L^{\infty} \rightarrow L^{\infty}}  & \leq \left\| {m^{\mathrm{LC}}}\right\|_{L^{\infty}\left(\Omega_1\right)}  \iint_{\Omega_1} \frac{\left|\bar{\partial}\mathcal{R}^{(2)} e^{2 i \theta}\right|}{|s-z|} d A(s) \\
    \nonumber
    &\leq \left\|\delta^{-2}\right\|_{L^{\infty}\left(\Omega_1\right)}\norm{m^\mathrm{LC}}{L^\infty}^2
    \left( \iint_{\Omega_1} \frac{\left|r^{\prime}\right| e^{-tz_0 u w}}{|s-z|} d A(s) + \iint_{\Omega_1} \frac{\left|s-z_0\right|^{-1 / 2} e^{-tz_0 u v}}{|s-z|} d A(s) \right)\\
    \nonumber
    &\leq (z_0t)^{-1/4}\left\|\delta^{-2}\right\|_{L^{\infty}\left(\Omega_1\right)}\norm{m^\mathrm{LC}}{L^\infty}^2
    \left( c\norm{r}{H^1} +c'\right)\\
    \nonumber
    &\leq  (z_0t)^{-1/4}C_S(\rho, \lambda)
\end{align}
where $\lambda:=\norm{r}{H^1}$.
So for sufficiently large $z_0t$, it is possible to invert the operator \eqref{op:S} by Neumann series. 
\begin{proposition}
    Given $r(z)\in  H^{1,2}_{\scriptscriptstyle 1/2}(\bbR)$ with $\|r\|_{H^{1,2}_{\scriptscriptstyle 1/2}} \leqslant \eta,\|r\|_{L^{\infty}} \leqslant \rho<1/2$ and $t^{1/4}>\sqrt{2}C_S$ , then
    \begin{equation}
     \left\vert  m^{(3)}(z)-I\right\vert\leq \frac{C_S}{(z_0t)^{1/4}-C_S} .
    \end{equation}
    As a consequence, 
    \begin{equation}
    \label{bd:m3}
         \left\vert  m^{(3)}(z)\right\vert \leq \frac{(z_0t)^{1/4}}{(z_0t)^{1/4}-C_S}.
          \end{equation}
\end{proposition}
We then obtain the following result that plays a key role in deriving the necessary \textit{a priori} estimates given by \cite{DZ-2}.
\begin{proposition}
\label{prop:mbound1}
   Given $r(z)\in  H^{1,2}_{\scriptscriptstyle 1/2}(\bbR)$ with $\|r\|_{H^{1,2}_{\scriptscriptstyle 1/2}} \leqslant \eta,\|r\|_{L^{\infty}} \leqslant \rho<1/2$ and $z_0t>4C_s^4>0$
    \begin{align}
        \norm{m_\pm(z)-I}{L^\infty}&\leq   \frac{c\rho (z_0t)^{1/4}}{(1-\rho)^2((z_0t)^{1/4}-C_S)}=C_m(\rho, \eta), \\
        \norm{\mu(z)-I}{L^\infty}&\leq  \frac{c \rho (z_0t)^{1/4}}{(1-\rho)^2((z_0t)^{1/4}-C_S)}=C_\mu(\rho, \eta).
        \end{align}
\end{proposition}
\begin{proof}
    A consequence of \eqref{bd:m3} and the $L^\infty$-bound of $m^{\textrm{LC}}$ (see Lemma \ref{lm:mlc}) and $\delta^{\sigma_3}$(see \cite[(2.2)]{CL} ).
\end{proof}
\begin{lemma}
\label{lm:mbound2}
   Given $r(z)\in  H^{1,2}_{\scriptscriptstyle 1/2}(\bbR)$ with $\|r\|_{H^{1,2}_{\scriptscriptstyle 1/2}} \leqslant \eta,\|r\|_{L^{\infty}} \leqslant \rho<1/2$ and $\sqrt{2}C_S^4>z_0t>0$
    \begin{align}
        \norm{m_\pm(z)}{L^\infty}&\leq C_{T m}(\eta), \\
        \norm{\mu(z)}{L^\infty}&\leq  C_{T\mu}(\eta).
    \end{align}
    \end{lemma}
\begin{proof}
Suppose that $t<T<+\infty$. In \cite{Zhou98}, it is proven that 
\begin{align}
    \norm{u(x,t)}{L^{2,1}}\leq C T \norm{r}{H^{1,2}}.
\end{align}
Thus by the standard Volterra theory, we can obtain that 
\begin{equation}
    \norm {m(z;x,t)}{L^\infty} \leq \exp\left(CT \norm{r}{H^{1,2}}\right):=C_{T m}(\eta).
\end{equation}
\end{proof}

\begin{lemma}
\label{lm:ED}
    Given   $r(z)\in  H^{1,2}_{\scriptscriptstyle 1/2}(\bbR)$ with $\|r\|_{H^{1,2}_{\scriptscriptstyle 1/2}} \leqslant \eta,\|r\|_{L^{\infty}} \leqslant \rho<1/2$ and $z_0t>16C_S^4>0$
    \begin{align}
        \norm{\Delta m^{(3)}(z)}{L^\infty}&\leq C_{\Delta m^{(3)}}(\rho, \eta)\norm{\Delta r}{H^{1,2}_1}.
    \end{align}
\end{lemma}

Notice that by \cite[(3.11)]{CL}, for $i=1,2$, we have
\begin{equation}
R_i(z)=\left(f_i(z)+\left[r_i(\operatorname{Re}(z))-f_i(z)\right] \cos(2\phi)\right) \delta(z)^{-2}
\end{equation}
and
\begin{equation}
\bar{\partial}\mathcal{R}_i^{(2)}\big\vert_{\Omega_1}=\bar{\partial} R_{i}=\frac{1}{2} r_i^{\prime}(\operatorname{Re} z) \cos(2\phi)\delta(z)^{-2}-\left[r_i(\operatorname{Re} z)-f_i(z)\right] \delta(z)^{-2} \frac{i e^{i \phi}}{|z-\xi|} \sin(2\phi)
\end{equation}

where
\begin{equation}
f_i(z)=r_i\left(z_0\right) e^{-2 \chi_i\left(z_0\right)} \eta_i\left(z ; z_0\right)^{-2} \delta_i(z)^2.
\end{equation}

We recall the definition of  $\delta$ from \cite[(2.2)]{CL}:

\begin{equation}
\delta(z)=\left(\frac{z-z_0}{z+z_0}\right)^{i \kappa} e^{\beta(z)}
\end{equation}
where 
\begin{align}
   \kappa&=-\frac{1}{2 \pi} \log \left(1-\left|r\left(z_0\right)\right|^2\right),\\
   \beta(z,z_0)&=\frac{1}{2 \pi i} \int_{-z_0}^{z_0} \log \left(\frac{1-|r(\zeta)|^2}{1-\left|r\left(z_0\right)\right|^2}\right) \frac{d \zeta}{\zeta-z},\\
   \left(\frac{z-z_0}{z+z_0}\right)^{i \kappa}&=\exp \left(i \kappa\left(\log \left|\frac{z-z_0}{z+z_0}\right|+i \arg \left(z-z_0\right)-i \arg \left(z+z_0\right)\right)\right).
\end{align}
We first recall the following result from \cite[Lemma 2.2 (iv)]{CL}: Along any ray of the form $\pm z_0+e^{i \phi} \mathbb{R}^{+}$ with $0<\phi<\pi$ or $\pi<\phi<2 \pi$,
\begin{equation}
    \left|\delta(z)-\left(\frac{z-z_0}{z+z_0}\right)^{i \kappa} e^{\beta\left( \pm z_0\right)}\right| \leq C_r\left|z \mp z_0\right|^{1 / 2} .
\end{equation}
The implied constant depends on $r$ through its $H^1(\bbR)$-norm and is independent of $z_0\in \bbR^+$.
We  recall the following useful properties from \cite{DZ-2}:
\begin{equation}
|\log(1-|r(z)|^2)| \leq \frac{|r(z)|^2}{1-|r(z)|^2}.
\end{equation}

\begin{equation}
\| \log(1-|r(z)|^2)\|_{L^2} \leq \frac{\|r\|_{L^\infty}\|r\|_{L^2}}{1-\|r\|_{L^\infty}}.
\end{equation}

\begin{equation}
\left\vert \log \left( \frac{1-|r_2(z)|^2}{1-|r_1(z)|^2}\right)\right\vert \leq \frac{2\rho}{1-\rho}|r_2(z)-r_1(z)|.
\end{equation}

\begin{equation}
   \left \Vert \left(\log \frac{1-\left\vert r_2(z)\right\vert^2}{1-\left\vert r_1(z)\right\vert^2}-\log \frac{1-\left\vert r_2(z_0)\right\vert^2}{1-\left\vert r_1(z_0)\right\vert^2}\right) \chi \right \Vert_{L^2} \leq \frac{c\rho}{1-\rho}\| r_2 - r_1\|_{H^{1}}.
\end{equation}

\begin{align}
\left\Vert \frac{d}{ds}\left\{ \log (1-|r(s)|^{2})-\log (1-|r(z_{0})|^{2})\right\} \chi _{(-z_0 ,z_{0})}\right\Vert _{L^{2}} &\leq \left\Vert \frac{r^{\prime }\bar{r}+r\bar{r}^{\prime }}{1-\left\vert
r\left( s\right) \right\vert ^{2}}\right\Vert
_{L^{2}} \\
\nonumber
&\leq \frac{c}{1-\rho} \|r\|_{H^{1}}.
\end{align}

\begin{align}
&\left\Vert \frac{d}{ds} \left(\left\{ \log\frac{ 1-|r_{2}(s)|^{2}}{
1-|r_{1}(s)|^{2}}-\log\frac{ 1-|r_{2}(z_{0})|^{2}}{
1-|r_{1}(z_{0}|^{2})}\right\} \chi _{(-z_0
,z_{0})}\right)\right\Vert _{L^2} \\
\nonumber
&\leq \frac{1}{1-\rho }\left\Vert r_{2}^{\prime }\bar{r}_{2}+r_{2}\bar{r}%
_{2}^{\prime }-r_{1}^{\prime }\bar{r}_{1}-r_{1}\bar{r}_{1}^{\prime
}\right\Vert _{L^{2}}\\
\nonumber
&\leq \frac{c}{(1-\rho)^2 }\left\Vert r_{2}-r_{1}\right\Vert _{H^{1}}.
\end{align}

We then show the following useful properties:
\begin{lemma}
    \begin{align}
    \|\beta(\cdot,z_0)\|_{L^{\infty}} &\leq c\frac{\|r\|_{H^{1}}}{1-\rho},\\
\|\Delta\beta(\cdot,z_0)\|_{L^{\infty}} &\leq c\frac{\|r_2-r_1\|_{H^{1}}}{1-\rho},\\
\|\beta(z,z_0)-\beta(z_0,z_0)\|_{L^{\infty}} &\leq c\frac{\|r\|_{H^{1}}}{1-\rho}(z-z_0)^{1/2},\\
\|\Delta(\beta(z,z_0)-\beta(z_0,z_0))\|_{L^{\infty}} & \leq c\frac{\|r_2-r_1\|_{H^{1}}}{1-\rho}(z-z_0)^{1/2},\\
\left|\left(\frac{z-z_0}{z+z_0}\right)^{i \kappa}\right| &\leq e^{\pi \kappa}\\
\left|\Delta\left(\frac{z-z_0}{z+z_0}\right)^{i \kappa}\right| &\leq c e^{\pi \kappa} \frac{\|r_2-r_1\|_{H^{1}}}{1-\rho} \ln \left\vert \frac{z-z_0}{z+z_0}\right\vert
\end{align}
\end{lemma}
\begin{proof}
According to the Lemma 23.3 in \cite{BDT}, we can  deduce that

\begin{align}
|\beta(z,z_0)| &\leq \left\Vert \left\{ \log(1-|r(s)|^2)-\log(1-|r(z_0)|^2) \right\} \chi_{(-z_0,z_0)}\right\Vert_{H^{1}}\\
\nonumber
&\leq \frac{c}{1-\rho} \|r\|_{H^{1}},
\end{align} 
and
\begin{align}
|\beta(z,z_0)-\beta(z_0,z_0)| &\leq \sqrt{2} \left\Vert \frac{d}{ds}\left(\left\{ \log(1-|r(s)|^2)-\log(1-|r(z_0)|^2)\right\} \chi_{(-z_0,z_0)}\right)\right\Vert_{L^2}|z-z_0|^{1/2}\\
\nonumber
&\leq \frac{c}{(1-\rho)^2} \|r\|_{H^{1}} |z-z_0|^{1/2}.
\end{align}
\begin{align}
|\Delta\left(\beta(z,z_0)-\beta(z_0,z_0)\right)| &\leq \sqrt{2} \left\Vert \frac{d}{ds} \left(\left\{ \log\frac{ 1-|r_{2}(s)|^{2}}{
1-|r_{1}(s)|^{2}}-\log\frac{ 1-|r_{2}(z_{0})|^{2}}{
1-|r_{1}(z_{0}|^{2})}\right\} \chi _{(-z_0
,z_{0})}\right)\right\Vert _{L^2}|z-z_0|^{1/2}\\
\nonumber
&\leq \frac{c}{(1-\rho)^2} \|\Delta r\|_{H^{1}} |z-z_0|^{1/2}.
\end{align}
\end{proof}
As a consequence, we have that
\begin{equation}
\label{est:Ddelta}
|\Delta \delta^{-2}| \leq c\left (1+\ln \left | \frac{z-z_0}{z+z_0}\   \right |  \right )\left \| r_2-r_1 \right \|_{H^{1}} .
\end{equation}

Given
\begin{equation}
\Delta W=( \Delta\bar{\partial} R )\delta^{-2}e^{2it\theta}+ \bar{\partial} R (\Delta \delta^{-2})e^{2it\theta}
\end{equation}
we deduce that
\begin{align}
    \Delta  \left( \left[r(\operatorname{Re} z)-f(z)\right] \delta(z)^{-2} \right) &= \left(\Delta \left[r(\operatorname{Re} z)-f(z)\right] \right)\delta(z)^{-2}+ \left[r(\operatorname{Re} z)-f(z)\right] \Delta\delta(z)^{-2}.
\end{align}
We only consider the following estimates:
\begin{align}
  \Delta \left( \left[r( z_0)-f(z)\right] \delta(z)^{-2}\right)&=\Delta\left( r(z_0) \delta^{-2}- r\left(z_0\right) e^{-2 \chi\left(z_0\right)} \eta\left(z ; z_0\right)^{-2}\right)\\
  \nonumber
  &=(\Delta r(z_0)) \left( \delta^{-2}-e^{-2 \chi\left(z_0\right)} \eta\left(z ; z_0\right)^{-2}\right)+ r(z_0)\Delta \left(\delta^{-2}-e^{-2 \chi\left(z_0\right)} \eta\left(z ; z_0\right)^{-2}\right).
\end{align}
Notice that 
\begin{align}
    \Delta \left(\delta^{-2}-e^{-2 \chi\left(z_0\right)} \eta\left(z ; z_0\right)^{-2}\right)&=\left(\Delta\eta\left(z ; z_0\right)^{-2}\right)\left(e^{-2\beta(z)}-e^{-2\beta(z_0)}\right) + \eta\left(z ; z_0\right)^{-2}\Delta\left(e^{-2\beta(z)}-e^{-2\beta(z_0)}\right).
\end{align}
We obtain
\begin{align}
    \left\vert \Delta \left(\delta^{-2}-e^{-2 \chi\left(z_0\right)} \eta\left(z ; z_0\right)^{-2}\right)\right\vert &\leq c e^{\pi \kappa} \frac{\|r_2-r_1\|_{H^{1}}}{1-\rho} \ln \left\vert \frac{z-z_0}{z+z_0}\right\vert \frac{c}{(1-\rho)^2} \|r\|_{H^{1}} |z-z_0|^{1/2}\\
    \nonumber
   & \quad +e^{\pi \kappa} \frac{c}{(1-\rho)^2} \|\Delta r\|_{H^{1}} |z-z_0|^{1/2}.
\end{align}

Then we can get the following estimate:
\begin{eqnarray}
\label{est:dR}
|\bar{\partial} \Delta R| &\leq& \frac{c_1}{|z-z_0|}\left( |\Delta(r-r(z_0))|+|\Delta(r(z_0)-f_1)| \right)+c_2|r_1'-r_2'|\\
\nonumber
&\leq& C_1 \| r_2-r_1 \|_{H^{1}}|z-z_0|^{-1/2}+C_2 \| r_2-r_1 \|_{H^{1}}+C_3 |r_2'-r_1'|
\end{eqnarray}

\begin{proof}[Proof of Lemma \ref{lm:ED}]
By the second resolvent identity, 
    \begin{equation}
   \Delta m^{(3)}=\Delta \left[\left({I-S}\right)^{-1}I \right]=(I+S_1+S_1^2+\cdots)(S_2-S_1)(I+S_2+S_2^2+\cdots)
\end{equation}
where for $i=1,2$
\begin{equation}
S_i[f]=\frac{1}{\pi} \int_{\mathbb{C}} \frac{f W_i}{s-z} d A(s).
\end{equation}
We then proceed to calculate $\Delta S(f)$:

\begin{equation}
|\Delta S(f)| \leq c(P_1+P_2)
\end{equation}
and will show that
\begin{eqnarray}
\label{P_1}
P_1 &=& \left \vert \int_{\mathbb{C}} \frac{f \bar{\partial} \Delta R_1 \delta^{-2}e^{2it\theta}}{s-z} d A(s) \right \vert \\
\nonumber
&\leq& C\|f\|_{L^\infty}\|\delta^{-2}\|_{L^{\infty}}\frac{\|\Delta r\|_{H^{1,1}}}{(z_0t)^{1/4}},
\end{eqnarray}

\begin{eqnarray}
\label{P2}
P_2 &=& \left \vert \int_{\mathbb{C}} \frac{f \bar{\partial} R_1 \Delta \delta^{-2}e^{2it\theta}}{s-z} d A(s) \right \vert\\
\nonumber
& \leq & C\|f\|_{L^\infty}\norm{r'}{L^2}\frac{\left \| r_2-r_1 \right \|_{H^{1}}}{(z_0t)^{1/4}}.
\end{eqnarray}
The estimate of \eqref{P_1} is equivalent to that of \eqref{op:S}. We will only focus on \eqref{P2}. 
Recall \eqref{est:Ddelta}. For brevity we only show the following calculation. For $|z|<1$,
\begin{align}
    \left\|\frac{\log |s| }{\sqrt{|s|}}\right\|_{L^p(v, +\infty)}&\leq \left(\int_{v}^{+\infty}\frac{|\log v|^p}{\left(u^2+v^2\right)^{p / 4}} d u\right)^{1 / p} \\
    \nonumber
    &\leq v^{1 / p-1 / 2}\left(\int_1^{\infty} \frac{|\log v|^p}{\left(1+w^2\right)^{p / 4}} d w\right)^{1 / p}\\
    \nonumber
    &\leq C v^{1 / p-1 / 2}|\log v|.
\end{align}
Also recall that 
\begin{align*}
    \left\|\frac{1}{|s-z|}\right\|_{L^q(v, +\infty)}&=\left(\int_{v}^{+\infty} \frac{1}{\left((u-\alpha)^2+(v-\beta)^2\right)^{q / 2}} d u\right)^{1 / q}\\
    &\leq\left(\int_{\mathbb{R}} \frac{1}{\left(s^2+(v-\beta)^2\right)^{q / 2}} d s\right)^{1 / 2}\\
    &\leq  c|v-\beta|^{1 / q-1} .
\end{align*}
We now combine the above estimates to see that
\begin{align*}
   & \int_0^{\infty} e^{-z_0t v^2}\left\|\frac{\log s}{\sqrt{|s|}}\right\|_{L^p(v, +\infty)}\left\|\frac{1}{|s-z|}\right\|_{L^q(v, +\infty)} d v\\
    &\quad \leq  \int_0^{\infty} e^{-z_0t v^2} v^{1 / p-1 / 2}|\log v||v-\beta|^{1 / q-1} d v\\
&= \int_0^\beta e^{-z_0t v^2}|\log( v)| v^{1 / p-1 / 2}(\beta-v)^{1 / q-1} d v\\
    &\quad +\int_\beta^{\infty} e^{-z_0t v^2}|\log( v)| v^{1 / p-1 / 2}(v-\beta)^{1 / q-1} d v\\
    &=\mathcal{I}_1+\mathcal{I}_2.
\end{align*}
For $\beta<1$, 
\begin{align*}
    \mathcal{I}_1&\leq \int_0^1 \sqrt{\beta} e^{-z_0t \beta^2 w^2} |\log(w)| w^{1 / p-1 / 2}(1-w)^{1 / q-1} d w+\int_0^1 \sqrt{\beta} e^{-z_0t \beta^2 w^2} |\log\beta| w^{1 / p-1 / 2}(1-w)^{1 / q-1} d w\\
    & \leq  c (z_0t)^{-1 / 4} \left( \int_0^1 |\log(w)| w^{1 / p-1}(1-w)^{1 / q-1} d w \right)+ c (z_0t)^{-1 / 8}\left(\beta^{1/4}\log\beta \int_0^1 w^{1 / p-3/4}(1-w)^{1 / q-1} d w \right)\\
    &\leq C (z_0t)^{-1/8}.
\end{align*}
For $\beta>1$, we make use of the fact that $\log\beta<\beta^{1/8}$ to deduce that 
\begin{align*}
     \mathcal{I}_1& \leq c(z_0t)^{-1/4}+\int_0^1 {\beta}^{5/8} e^{-z_0t \beta^2 w^2} w^{1 / p-1 / 2}(1-w)^{1 / q-1} d w\\
     &\leq c(z_0t)^{-1/4}+c(z_0t)^{-5/16}\int_0^1   w^{1 / p-9/ 8}(1-w)^{1 / q-1} d w\\
     &\leq C(z_0t)^{-1/4}.
\end{align*}
Moreover,
\begin{align*}
    \mathcal{I}_2&=\int_0^{+\infty} e^{-z_0t(w+\beta)^2}|\log(w+\beta )|(w+\beta)^{1 / p-1 / 2} w^{1 / q-1} d w\\
     &=\int_0^{1-\beta} e^{-z_0t(w+\beta)^2}|\log(w+\beta )|(w+\beta)^{1 / p-1 / 2} w^{1 / q-1} d w+ \int_{1-\beta}^{+\infty} e^{-z_0t(w+\beta)^2}|\log(w+\beta )|(w+\beta)^{1 / p-1 / 2} w^{1 / q-1} d w\\
   &\leq   \int_0^{1-\beta} e^{-z_0t w^2}|\log(w )| w^{-1/2} d w+\int_{1-\beta}^{+\infty} e^{-z_0t(w+\beta)^2}(w+\beta)^{1 / p+1 / 2}w^{1 / q-1} dw\\
   &\leq \int_0^{+\infty} e^{-z_0t w^2}|\log(w )| w^{-1/2} d w+\norm{ e^{-z_0t(w+\beta)^2/2}(w+\beta)}{L^\infty}\int_{0}^{+\infty}e^{-z_0t w^2/2}w^{1/q-1}dw\\
   &\leq c \log(z_0 t)/(z_0t)^{1/4}+ c (z_0t)^{-1/2-1/2q}.
\end{align*}


\end{proof}

\begin{lemma}
\label{lm:deltaM}
   Given $r(z)\in  H^{1,2}_{\scriptscriptstyle 1/2}(\bbR)$ with $\|r\|_{H^{1,2}_{\scriptscriptstyle 1/2}} \leqslant \eta,\|r\|_{L^{\infty}} \leqslant \rho<1/2$ and $x/t^{1/3} \leq -1$
    \begin{align}
        \left\vert \Delta m_\pm(z)\right\vert&\leq C_{\Delta m}^2(\eta) \left | 1+\ln \left | \frac{z-z_0}{z+z_0}\ \right |  \right |\norm{\Delta r}{H^{1,2}_{\scriptscriptstyle 1/2}},\\
        \left\vert\Delta \mu(z)\right\vert&\leq C_{\Delta \mu}^2 (\eta) \left | 1+\ln \left | \frac{z-z_0}{z+z_0}\ \right |  \right |\norm{\Delta r}{H^{1,2}_{\scriptscriptstyle 1/2}}.
    \end{align}  
    
\end{lemma}
\begin{proof}
We first recall from \ref{m:trans} that 
\begin{align*}
  \Delta m(z ; x, t)&=\Delta m^{(3)}(z ; x, t) m^{\mathrm{LC}}\left(z ; z_0\right) \mathcal{R}^{(2)}(z)^{-1} \delta(z)^{\sigma_3}\\
  &\quad + m^{(3)}(z ; x, t) \Delta m^{\mathrm{LC}}\left(z ; z_0\right) \mathcal{R}^{(2)}(z)^{-1} \delta(z)^{\sigma_3}\\
  &\quad + m^{(3)}(z ; x, t)  m^{\mathrm{LC}}\left(z ; z_0\right) \Delta\mathcal{R}^{(2)}(z)^{-1} \delta(z)^{\sigma_3}\\
  &\quad + m^{(3)}(z ; x, t)  m^{\mathrm{LC}}\left(z ; z_0\right) \mathcal{R}^{(2)}(z)^{-1} \Delta\delta(z)^{\sigma_3}.
\end{align*}
 The proof is a combination of bounds given by \eqref{est:Ddelta}, the explicit form of $\mathcal{R}^{(2)}$ given by \cite[(3.3)-(3.10)]{CL}, Lemma \ref{lm:mlc} and Lemma \ref{lm:mbound2}.
\end{proof}
\subsection{the analysis of $\textrm{R}_2$ and $\textrm{R}_3$ }
For $\textrm{R}_2$, we have that 
$$
|z_0|=\sqrt{\frac{|x|}{12t}}\lesssim t^{-1 / 3} \rightarrow 0 \quad \text { as } t \rightarrow \infty,
$$
we do not need the lower/upper factorization of the jump matrix for $|z|<z_0$ as is given by \cite[(2.4)]{CL}. Instead we only deal with the following factorization:
\begin{equation}
    e^{-i \theta \operatorname{ad} \sigma_3} v(z)=\twomat{1}{-\overline{r(z)} e^{-2 i \theta}}{0}{1}\twomat{1}{0}{r(z) e^{2 i \theta}}{1}.
\end{equation}
We then carry out the following scaling:
$$
z \rightarrow \zeta t^{-1 / 3}
$$
and the factorization becomes
\begin{equation}
    \twomat{1}{-\overline{r\left(\zeta t^{-1 / 3}\right)} e^{-2 i \theta\left(\zeta t^{-1 / 3}\right)}}{0}{1}\twomat{1}{0}{r\left(\zeta t^{-1 / 3}\right) e^{2 i \theta\left(\zeta t^{-1 / 3}\right)}}{1}
\end{equation}
where
\begin{equation}
    \theta\left(\zeta t^{-1 / 3}\right)=4 \zeta^3+x \zeta t^{-1 / 3}.
\end{equation}
We then proceed to verify that along $\Sigma_1^{(\textrm{P})}\cup \Sigma_2^{(\textrm{P})}$
\begin{align}
 \operatorname{Re}\left(2 i \theta\left(\zeta t^{-1 / 3}\right)\right)&=-24 u^2 v+8 v^3-2 vxt^{-1/3}\\
 \nonumber
 &\leq-v(16u^2+2xt^{-1/3})\\
 \nonumber
 &\leq -v(16u^2-2).
\end{align}
So we can reduce the original RHP
to a problem on the contour given by Figure \ref{fig: Painleve} with the following jumps:
\begin{align}
    e^{-i \theta \operatorname{ad} \sigma_3} v^{(2)}(\zeta)&=e^{-4 i\left(\zeta^3+\left(x /\left(4 t^{1 / 3}\right)\right) \zeta\right) \text { ad } \sigma_3}\twomat{1}{0}{r(0)}{1}, \quad \zeta \in \Sigma_1^{(\mathrm{P})} \cup \Sigma_2^{(\mathrm{P})},\\
    \nonumber
    &=e^{-4 i\left(\zeta^3+\left(x /\left(4 t^{1 / 3}\right)\right) \zeta\right) \text { ad } \sigma_3}\twomat{1}{-\overline{r(0)}}{0}{1}, \quad \zeta \in \Sigma_3^{(\mathrm{P})} \cup \Sigma_4^{(\mathrm{P})}.
\end{align}
and the interpolation between $\mathbb{R}$ and $\Sigma_1^{(\mathrm{P})} \cup \Sigma_2^{(\mathrm{P})}$ is given by
$$
r\left(0\right)+\left(r\left(\operatorname{Re} \zeta t^{-1 / 3}\right)-r\left(0\right)\right) \cos 2 \phi
$$
Following the conventions in \cite[Section 7]{CL},  transformations are made to the solution $m(z; x, t)$ to Problem \ref{prob:mKdV.RH0}
\begin{equation}
\label{m:trans1}
    m(z ; x, t)=m^{(3)}(z ; x, t) m^{\mathrm{LC}}\left(z ; z_0\right) \mathcal{R}^{(2)}(z)^{-1}.
\end{equation}
Here $\mlc$ is explicitly solvable by the Painlev\'e II equation and $\mathcal{R}^{(2)}(z)$ is given by:
$$
\mathcal{R}^{(2)}(z)= \begin{cases} \twomat{1}{0}{-r\left(\zeta t^{-1 / 3}\right) e^{2 i \theta\left(\zeta t^{-1 / 3}\right)}}{1}  & \zeta \in\mathbb{R} ,\\ \twomat{1}{0}{-r\left(0\right) e^{2 i \theta\left(\zeta t^{-1 / 3}\right)}}{1} & \zeta \in \Sigma^{(\textrm{P})}_1\cup \Sigma^{(\textrm{P})}_2;\end{cases}
$$
$$
\mathcal{R}^{(2)}(z)= \begin{cases} \twomat{1}{-\overline{r\left(\zeta t^{-1 / 3}\right)} e^{-2 i \theta\left(\zeta t^{-1 / 3}\right)}}{0}{1}  & \zeta \in\mathbb{R}, \\ \twomat{1}{-\overline{r\left(0\right)} e^{-2 i \theta\left(\zeta t^{-1 / 3}\right)}}{0}{1} & \zeta \in \Sigma^{(\textrm{P})}_3\cup \Sigma^{(\textrm{P})}_4.\end{cases}
$$
Following the proof of Lemma \ref{lm:deltaM}, we can also deduce the following lemma:
\begin{lemma}
\label{lm:deltaM'}
   Given $r(z)\in  H^{1,2}_{\scriptscriptstyle 1/2}(\bbR)$ with $\|r\|_{H^{1,2}_{\scriptscriptstyle 1/2}} \leqslant \eta,\|r\|_{L^{\infty}} \leqslant \rho<1/2$ and $|x/t^{1/3}|<1$
    \begin{align}
        \left\vert \Delta m_\pm(z)\right\vert&\leq C_{\Delta m}^2(\eta) \norm{\Delta r}{H^{1,2}_{\scriptscriptstyle 1/2}},\\
        \left\vert\Delta \mu(z)\right\vert&\leq C_{\Delta \mu}^2 (\eta) \norm{\Delta r}{H^{1,2}_{\scriptscriptstyle 1/2}}.
    \end{align}  
    
\end{lemma}
For $\mathrm{R}_2$, we have that $x/t^{1/3}>1$. More specifically, 
\begin{align}
\operatorname{Re}\left(2 i \theta\left(\zeta z_0\right)\right) & =8 \tau\left(-3u^2 v+v^3+3 v\right) \\
\nonumber
& \leq 8 \tau v\left(-2 u^2 +3\right) 
\end{align}
So we can follow the proof of Lemma \ref{lm:deltaM'} to deduce the following lemma. 
\begin{lemma}
\label{lm:deltaM''}
   Given $r(z)\in  H^{1,2}_{\scriptscriptstyle 1/2}(\bbR)$ with $\|r\|_{H^{1,2}_{\scriptscriptstyle 1/2}} \leqslant \eta,\|r\|_{L^{\infty}} \leqslant \rho<1/2$ and $x/t^{1/3}>1$
    \begin{align}
        \left\vert \Delta m_\pm(z)\right\vert&\leq C_{\Delta m}^2(\eta) \norm{\Delta r}{H^{1,2}_{\scriptscriptstyle 1/2}},\\
        \left\vert\Delta \mu(z)\right\vert&\leq C_{\Delta \mu}^2 (\eta) \norm{\Delta r}{H^{1,2}_{\scriptscriptstyle 1/2}}.
    \end{align}  
\end{lemma}
For $\mathrm{R}_3$, we have that $x/t^{1/3}>1$. More specifically, 
\begin{align}
 \operatorname{Re}\left(2 i \theta\left(\zeta t^{-1 / 3}\right)\right)&=-24 u^2 v+8 v^3-2 vxt^{-1/3}\\
 \nonumber
 &\leq-v(16u^2+2xt^{-1/3})\\
 \nonumber
 &\leq -v(16u^2+2).
\end{align}
So we can follow the proof of Lemma \ref{lm:deltaM'} to deduce the following lemma. 
\begin{lemma}
\label{lm:deltaM'''}
   Given $r(z)\in  H^{1,2}_{\scriptscriptstyle 1/2}(\bbR)$ with $\|r\|_{H^{1,2}_{\scriptscriptstyle 1/2}} \leqslant \eta,\|r\|_{L^{\infty}} \leqslant \rho<1/2$ and $x/t^{1/3}>1$
    \begin{align}
        \left\vert \Delta m_\pm(z)\right\vert&\leq C_{\Delta m}^2(\eta) \norm{\Delta r}{H^{1,2}_{\scriptscriptstyle 1/2}},\\
        \left\vert\Delta \mu(z)\right\vert&\leq C_{\Delta \mu}^2 (\eta) \norm{\Delta r}{H^{1,2}_{\scriptscriptstyle 1/2}}.
    \end{align}  
    
\end{lemma}
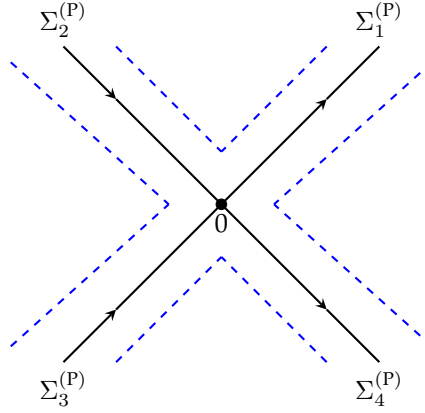
\begin{figure}[H]
\caption{$\Sigma$-Painlev\'e}
\vskip 15pt
\begin{tikzpicture}[scale=0.7]
\draw[thick]		(3, 3) -- (2,2);						
\draw[->,thick,>=stealth] 		(0,0) -- (2,2);		
\draw[thick] 			(0,0) -- (-2,2); 				
\draw[->,thick,>=stealth]  	(-3,3) -- (-2,2);	
\draw[->,thick,>=stealth]		(-3,-3) -- (-2,-2);							
\draw[thick]						(-2,-2) -- (0,0);
\draw[thick,->,>=stealth]		(0,0) -- (2,-2);								
\draw[thick]						(2,-2) -- (3,-3);
\draw	[fill]						(0,0) circle[radius=0.1];	
\node [below] at (0,0) {0};

\draw [dashed, thick, blue] (2,3)--(0,1);
\draw [dashed, thick, blue] (-2,3)--(0,1);
\draw [dashed, thick, blue] (2,-3)--(0,-1);
\draw [dashed, thick, blue] (-2,-3)--(0,-1);

\draw [dashed, thick, blue] (-4, 2.7)--(-1,0);
\draw [dashed, thick, blue] (-4,-2.7)--(-1,0);
\draw [dashed, thick, blue] (4, 2.7)--(1,0);
\draw [dashed, thick, blue] (4,-2.7)--(1,0);

\node[above] at (3,3)					{$\Sigma^{(\text{P})}_1$};
\node[above] at (-3,3)					{$\Sigma^{(\text{P})}_2$};
\node[below] at (-3,-3)					{$\Sigma^{(\text{P})}_3$};
\node[below] at (3,-3)				{$\Sigma^{(\text{P})}_4$};
\end{tikzpicture}
\label{fig: Painleve}
\end{figure}
As a consequence, we obtain the following proposition which is used in the proof of Theorem \ref{thm:r}.
\begin{proposition}
\label{prop:new apriori}
  Suppose $r \in H_{\scriptscriptstyle 1/2}^{1,2},\|r\|_{H_{\scriptscriptstyle 1/2}^{1,1}} \leqslant \eta,\|r\|_{L^{\infty}} \leqslant \rho<1/2$. Then for any $x, t \in \mathbb{R}$, and $2<p<\infty,\left(1-C_{v_\theta}\right)^{-1}$  exists as bounded operators in $L^p(\mathbb{R})$ and satisfy the bounds

$$
\left\|\left(1-C_{v_\theta}\right)^{-1}\right\|_{L^p \rightarrow L^p} \leqslant\mathcal{K}_p,
$$
where
\begin{equation}
\label{Kp}
\mathcal{K}_p(\rho, \eta)=(1+\rho)\left(\mathbf{c}_p \left(1+\rho\right)M_{\infty}^2+1\right). 
\end{equation}
The constant $\mathbf{c}_p$ is given by the $L^p$-boundedness of the Cauchy projection.
\end{proposition}
\begin{proof}
Assuming that $m_\pm \in L^\infty$, then by \cite[Proposition 4.5, Proposition 2.6]{DZ-3}, the operator $ \left(I-C_{v_\theta}\right)^{-1} $ exists on $L^p$, and we also recall the following fact from \cite[(7.107)]{Deift}:
    \begin{equation}
        \left(I-C_{v_\theta}\right)^{-1} g=\left(C_{+}\left(g(v_\theta-I) m_{+}^{-1}\right)\right) m_{+} v_{+}^{-1}+g v_{+}^{-1}.
    \end{equation}
    For any $g\in L^p(\bbR)$, from
    \begin{align*}
        \norm{ \left(I-C_{v_\theta}\right)^{-1} g}{L^p}&\leq \norm{\left(C_{+}\left(g(v_\theta-I) m_{+}^{-1}\right)\right) m_{+} v_{+}^{-1}}{L^p}+\norm{g v_{+}^{-1}}{L^p}\\
        &\leq \norm{\left(C_{+}\left(g(v_\theta-I) m_{+}^{-1}\right)\right) }{L^p}\norm{m_{+} v_{+}^{-1}}{L^\infty}+\norm{g v_{+}^{-1}}{L^p}\\
        &\leq \mathbf{c}_p \norm{g}{L^p}\norm{(v_\theta-I) m_{+}^{-1}}{L^\infty} \norm{m_{+} v_{+}^{-1}}{L^\infty}+\norm{g v_{+}^{-1}}{L^p},
    \end{align*}
    we deduce that
    \begin{align}
        \frac{\norm{ \left(I-C_{v_\theta}\right)^{-1} g}{L^p}}{\norm{g}{L^p}}&\leq \mathbf{c}_p \norm{(v_\theta-I) m_{+}^{-1}}{L^\infty} \norm{m_{+} v_{+}^{-1}}{L^\infty}+\norm{v_{+}^{-1}}{L^\infty}\\
        \nonumber
        &\leq \mathbf{c}_p \norm{v_\theta-I}{L^\infty}\norm{v_+}{L^\infty}\norm{m_+}{L^\infty}^2+\norm{v_{+}^{-1}}{L^\infty}\\
        \nonumber
        &\leq (1+\rho)\left(\mathbf{c}_p\norm{m_+}{L^\infty}^2\left(1+\rho\right)+1\right).
    \end{align}
    The conclusion follows if we replace $\norm{m_+}{L^\infty}$ by 
    \begin{equation}
    \label{Minfty}
        M_{\infty}(\rho,\eta)=\text{max}\lbrace  C_{T m}(\eta), C_m(\rho, \eta)+1\rbrace
    \end{equation}
   where $C_m(\rho, \eta)$ and $C_{T m}(\eta)$ are given in Proposition \ref{prop:mbound1} and Lemma \ref{lm:mbound2} respectively.
\end{proof}

\subsection{Supplementary estimates}
\label{sec:suppesti}
In this subsection we give some other the necessary estimates needed to obtain \eqref{est:F}-\eqref{est:DF}. Throughout the section we assume that $r(z)\in H^{1,2}_1$ with $\|r\|_{H^{1,2}} \leqslant \eta$, $\|r\|_{L^{\infty}} \leqslant \rho<\frac{1}{2}$. We will make use of the  uniform $L^\infty$ bound of $m_\pm$ and $\mu$ obtained from the previous section. 
We first recall from \cite[Lemma 5.24]{DZ-2} that 
\begin{align}
    \|u\|_{L^2} &\leqslant \frac{1}{\sqrt{2 \pi}} \frac{\|r\|_{L^2}}{(1-\varrho)^{1/2}},
\\
    \left\| u_x\right\|_{L^2} &\leqslant \frac{1}{\sqrt{2 \pi}} \frac{\|\diamond r\|_{L^2}}{(1-\varrho)^{1 / 2}}.
\end{align}
For  $u_{xx}$, we recall from \cite[(2.26)]{Zhou98} that 

\begin{equation}
\label{L2 uxx}
u_{xx}=\frac{\operatorname{ad} \sigma}{2 \pi} \int\left( \sum_{j=0}^{1} c_j z^j \mu\left(w^-_{\theta}+w^+_{\theta}\right) +(i \operatorname{ad} \sigma)^2 z^2 \mu\left(w^-_{\theta}+w^+_{\theta}\right)\right)
\end{equation}

where $c_j$ are polynomials of $u$ and $u_x$. So we conclude that 
\begin{equation}
    \norm{ u_{xx}}{L^2} \leq c\left(\norm{u}{H^1}+\norm{\mu}{L^\infty}\norm{r}{L^{2,2}}\right).
\end{equation}
Due to the singularity $\ln|(z-z_0)/(z+z_0)|$ besides Lemma \ref{lm:deltaM}, we need another bound on $\Delta m_\pm$. We proceed to differentiate the jump relation $m_+=m_-v_\theta$, we deduce
\begin{equation}
    \partial_z m_{+}-\partial_z m_{-}=\left(\partial_z m_{-}\right) (v_\theta-I)+m_{-}\left(-i(x+12 t z^2) \operatorname{ad} v_\theta+\left(\partial_z v\right)_\theta\right).
\end{equation}
Setting $\mathcal{C}_{v_\theta} f= C^-f(v_\theta-I)$, we obtain 
\begin{equation}
   \partial_z m_{-}=\left(I-\mathcal{C}_{v_\theta}\right)^{-1}C_\bbR^-\left[ m_{-}\left(-i(x-12 t z^2) \operatorname{ad} v_\theta+\left(\partial_z v\right)_\theta\right) \right]. 
\end{equation}
It is easy to deduce that 
\begin{align*}
    \norm{\partial_z m_\pm}{L^2(dz)} &\leq  \norm{\left(I-\mathcal{C}_{v_\theta}\right)^{-1}}{L^2}\norm{ m_-}{L^\infty(dz)} \left(\left\vert x \right\vert  \norm{r}{L^{2}} +\left\vert 12t \right\vert   \norm{r}{L^{2,2}}  +\norm{r'}{L^{2}}\right),
\end{align*}
and consequently, 
\begin{align}
    \norm{m_\pm}{L^\infty(dz)} &\leq \norm{m_\pm}{H^1(dz)}\\
    \nonumber
    &\leq \norm{m_\pm}{L^2(dz)}+ \norm{\left(I-\mathcal{C}_{v_\theta}\right)^{-1}}{L^2}\norm{ m_-}{L^\infty(dz)} \left(\left\vert x \right\vert  \norm{r}{L^{2}} +\left\vert 12t \right\vert   \norm{r}{L^{2,2}}  +\norm{r'}{L^{2}}\right).
\end{align}
Moreover,  we have that 
\begin{align}
    \norm{\Delta m_\pm}{L^\infty(dz)}&\leq \norm{\Delta m_\pm}{L^2(dz)}+ \norm{\partial_z \Delta m_\pm}{L^2(dz)} \\
    \nonumber
    &\leq  \norm{\Delta (1-\mathcal{C}_{v_\theta})^{-1}}{L^2}\norm{ m_-}{L^\infty(dz)}\left(\left\vert x \right\vert  \norm{r}{L^{2}} +\left\vert 12t \right\vert   \norm{r}{L^{2,2}}  +\norm{r'}{L^{2}}\right)\\
    \nonumber
    &\quad + \norm{ (1-\mathcal{C}_{v_\theta})^{-1}}{L^2}\left(\left\vert x \right\vert  \norm{r \Delta  m_-}{L^{2}} +\left\vert 12t \right\vert   \norm{r \Delta  m_-}{L^{2,2}}  +\norm{r' \Delta  m_-}{L^{2}}\right)\\
    \nonumber
    &\quad +  \norm{ (1-\mathcal{C}_{v_\theta})^{-1}}{L^2}\norm{ m_-}{L^\infty(dz)}\left(\left\vert x \right\vert  \norm{\Delta r}{L^{2}} +\left\vert 12t \right\vert   \norm{r}{L^{2,2}}  +\norm{r'}{L^{2}}\right)\\
    \nonumber
    &\quad + \norm{ (1-\mathcal{C}_{v_\theta})^{-1}}{L^2}\norm{ m_-}{L^\infty(dz)}\left(\left\vert x \right\vert  \norm{ r}{L^{2}} +\left\vert 12t \right\vert   \norm{\Delta r}{L^{2,2}}  +\norm{r'}{L^{2}}\right)\\
    \nonumber
    &\quad + \norm{ (1-\mathcal{C}_{v_\theta})^{-1}}{L^2}\norm{ m_-}{L^\infty(dz)}\left(\left\vert x \right\vert  \norm{ r}{L^{2}} +\left\vert 12t \right\vert   \norm{ r}{L^{2,2}}  +\norm{\Delta r'}{L^{2}}\right)\\
    \nonumber
    & \leq  F(x,t,\eta)\norm{\Delta r}{H^{1,2}} + \norm{ (1-\mathcal{C}_{v_\theta})^{-1}}{L^2} \norm{  r' \Delta m_-}{L^{2}}\\
    \nonumber
    &\leq F(x,t,\eta)\norm{\Delta r}{H^{1,2}} + (1-\rho)^{-1} \norm{\Delta m_-}{L^\infty(dz)} \norm{  r'}{L^{2}}.
\end{align}
We can obtain the following bound as long as $(1-\rho)^{-1} \norm{  r'}{L^{2}}<1$:
\begin{equation}
\label{est:dm}
     \norm{\Delta m_\pm}{L^\infty(dz)} \leq \dfrac{z_0^2 F(x,t,\eta)\norm{\Delta r}{H^{1,2}} }{1-(1-\rho)^{-1}\norm{  r'}{L^{2}}}
\end{equation}
and it is easy to see that $(1-\rho)^{-1} \norm{  r'}{L^{2}}<1$ if we have $\norm{r}{H^1}<1/2$. We also point out that  $F(x,t,\eta)$ grows linearly in $x$ and $t$.

\section*{Acknowledgment}
The first author was partially supported by NSF grant DMS-2350301 and by Simons foundation MP-TSM-
00002258. The second author is partially supported by NSFC Grant No.12201605.

\bibliographystyle{amsplain}

\begin{thebibliography}{10}


\bibitem{BDT}
P.~R. Beals, P. Deift and C. Tomei, Direct and inverse scattering on the line, Number 28 in Mathematical surveys and monographs. AMS, Providence, R.I., {\bf 1988}; 

 \bibitem{CL}
 G. Chen and J. Liu, Long-time asymptotics of the modified KdV equation in weighted Sobolev spaces, Forum Math. Sigma {\bf 10} (2022), Paper No. e66, 52 pp.
 \bibitem{CLT}
 Gong Chen, Jiaqi Liu, Yuanhong Tian. Perturbation of the Nonlinear Schrodinger Equation by a Localized Nonlinearity. 2024. ⟨hal-04805270⟩
\bibitem{Deift}
P.~A. Deift,  Orthogonal polynomials and random matrices: a Riemann-Hilbert approach. Courant Lecture
Notes in Mathematics, 3. New York University, Courant Institute of Mathematical Sciences, New York;
American Mathematical Society, Providence, RI, 1999. viii+273 pp
\bibitem{DZ93}
Deift, P.,  Zhou, X.: A steepest descent method for oscillatory Riemann-Hilbert problems. Asymptotics for the mKdV equation. \emph{Ann. of Math.} (2) \textbf{137} (1993), 295--368. 
\bibitem{DZ-1}
P.~A. Deift and X. Zhou, \textit{a priori} $L^p$-estimates for solutions of Riemann-Hilbert problems, Int. Math. Res. Not. {\bf 2002}, no.~40, 2121--2154; 
 \bibitem{DZ-2}
 P.~A. Deift and X. Zhou, Perturbation theory for infinite-dimensional integrable systems on the line. A case study, Acta Math. {\bf 188} (2002), no.~2, 163--262; 
 \bibitem{DZ-3}
 P.~A. Deift and X. Zhou, Long-time asymptotics for solutions of the NLS equation with initial data in a weighted Sobolev space, Comm. Pure Appl. Math. {\bf 56} (2003), no.~8, 1029--1077; 
  \bibitem{DMM18}
Dieng, M., McLaughlin, K D.-T. : Long-time Asymptotics for the NLS equation via dbar methods. Preprint,
 arXiv:0805.2807, 2008.
 \bibitem{Engel}
 J. Engelbrecht, Solutions to the perturbed KdV equation, Wave Motion {\bf 14} (1991), no.~1, 85--92; MR1117439.

 \bibitem{EnKh} J. Engelbrecht and Y. Khamidullin, On the possible amplification of nonlinear seismic waves, Phys. Earth Plan. Int. {\bf 50}, 39-45 (1988).
 \bibitem{FIKN}
 Fokas, Athanassios S.; Its, Alexander R.; Kapaev, Andrei A.; Novokshenov, Victor Yu. Painlevé transcendents. The Riemann-Hilbert approach
Math. Surveys Monogr., 128
American Mathematical Society, Providence, RI, 2006. xii+553 pp.

\bibitem{GPR}
P. Germain, F. Pusateri and F. Rousset, Asymptotic stability of solitons for mKdV, Adv. Math. {\bf 299} (2016), 272--330; MR3519470
\bibitem{HGR}
B.~H. Harrop-Griffiths, Long time behavior of solutions to the mKdV, Comm. Partial Differential Equations {\bf 41} (2016), no.~2, 282--317.
 \bibitem{JLPS}
 R. Jenkins, J. Liu, P. Perry, C. Sulem, Soliton resolution for the derivative nonlinear Schrödinger equation, Commun. Math. Phys. 363 (3)
(2018) 1003–1049.
 \bibitem{Kaup}
 KAUP, D. J., A perturbation expansion for the Zakharov-Shabat inverse scattering 
transform. SIAM J. Appl. Math., 31 (1976), 121-133
\bibitem{KN}
KAUP, D. J. ~ NEWELL, A. C., Solitons as particles, oscillators, and in slowly changing media: a singular perturbation theory. Proc. Roy. Soe. London Ser. A, 361 
(1978), 413-446. 
\bibitem{LPS} 
J. Liu, P.~A. Perry and C. Sulem, Long-time behavior of solutions to the derivative nonlinear Schr\"odinger equation for soliton-free initial data, Ann. Inst. H. Poincar\'e{} C Anal. Non Lin\'eaire {\bf 35} (2018), no.~1, 217--265; MR3739932
 \bibitem{Zhou98}
X. Zhou, $L^2$-Sobolev space bijectivity of the scattering and inverse scattering transforms, Comm. Pure Appl. Math. {\bf 51} (1998), no.~7, 697--731; 
\end{thebibliography}

\end{document}